\documentclass{article}
\usepackage[UKenglish]{babel}
\usepackage{authblk}
\usepackage[colorlinks=true]{hyperref}
\hypersetup{citecolor=blue}
\usepackage{color}
\usepackage{amsfonts,amsthm,amssymb}
\usepackage{mathtools}
\usepackage{enumitem}
\usepackage{dsfont}
\usepackage{esint}

\usepackage[capitalise]{cleveref}
\usepackage{float}
\usepackage{epsfig}
\usepackage{graphicx}
\usepackage{caption}
\usepackage{subcaption}
\usepackage{tikz}
\newtheorem{remark}{Remark}
\newtheorem{thm}{Theorem}[section]
\newtheorem{lem}[thm]{Lemma}
\newtheorem{prop}[thm]{Proposition}
\newtheorem{cor}[thm]{Corollary}

\numberwithin{equation}{section}
\numberwithin{thm}{section}
\newcommand{\beq}{\begin{equation}}
\newcommand{\eeq}{\end{equation}}
\newcommand {\f}{\frac}
\newcommand {\pa}{\partial}
\newcommand {\e}{\varepsilon}

\newcommand {\init}{\text{in}}
\newcommand{\R}{\mathbb{R}}

\newcommand\N{{\mathbb N}}

\newcommand\MScN[1]{\href{http://www.ams.org/mathscinet-getitem?mr=#1}{\nolinkurl{(#1)}}}
\newcommand\DOI[1]{\href{http://dx.doi.org/#1}{(doi: \nolinkurl{#1})}}
\newcommand\LINK[1]{\href{#1}{(link: \nolinkurl{#1})}}
\title {Study of a class of triangular starvation driven
  cross-diffusion systems}

\author{E. Brocchieri$^1$, L. Desvillettes$^2$, H. Dietert$^3$}

\date{\today}
\providecommand{\keywords}[1]{\small\textit{{Keywords.}} #1\\}
\providecommand{\subjclass}[1]{\small\textit{{2020 Mathematics Subject Classification.}} #1}
\begin{document}
\maketitle
\begin{abstract}
  We study the existence, regularity and uniqueness for a general
  class of triangular reaction-cross-diffusion systems coming from
  the study of \textit{starvation driven} behavior for two species in
  competition. This study involves an equivalent system
  in \textit{non-divergence} form, for which existence can be obtained
  thanks to Schauder's fixed point theorem.
\end{abstract}
\keywords{Cross-diffusion, starvation-driven models, existence, uniqueness}
\subjclass{Primary 35K57, Secondary 92D25}

\section{Introduction}
This paper is devoted to the analysis of a class of triangular
cross-diffusion systems with unknowns $u=u(t,x)\ge 0$ and
$v=v(t,x) \ge0$, representing the densities of two populations. We
assume that the species represented by $v$ diffuses with a given
constant rate, while the species represented by $u$ diffuses with a
rate $B(u,v)$, which depends on both $u$ and $v$. In other terms, the
equation satisfied by $u$ involves a cross-diffusion term. The system
is said to be triangular as no cross-diffusion term appears in the
equation satisfied by $v$. Moreover, we include reaction terms
modeling the competition between the two species.  \medskip

We consider the evolution on a smooth ($C^{\infty}$) bounded domain
\(\Omega \subset \R^N\), \(N \ge 1\), over a time \(T\). Denoting
\(\Omega_T \coloneqq (0,T) \times \Omega\), the system writes
as
\begin{equation}\label{macro step 1}
  \begin{cases}
    \partial_t u=\Delta \big(u\,B(u,v)\big)+uf(u, v),\qquad &\text{in }\Omega_T,\\
    \partial_t v=d_v\Delta v+vg(u, v),&\text{in }\Omega_T.
  \end{cases}
\end{equation}
It is endowed with zero flux (homogeneous Neumann) boundary
conditions
\begin{equation}\label{boundary cond macro step 1}
  \nabla \big(u \,B(u,v)\big)\cdot\sigma=\nabla v\cdot\sigma=0,\qquad\text{on }\,(0,T)\times\partial\Omega,
\end{equation}
and with nonnegative initial data
\begin{equation}\label{initial data step1}
  u(0,x)=u_\init(x)\ge0,\qquad
  v(0,x)=v_\init(x)\ge0,\qquad x\in\,\Omega.
\end{equation}
In the sequel, we will consider a diffusion coefficient $d_v$ and
functions $B,f,g$ which satisfy the following assumptions.  \medskip

\noindent\textbf{Assumption A}. The diffusion rate associated to the
species $v$ is strictly positive, i.e.
\begin{equation*}
  d_v > 0.
\end{equation*}
The functions $f,g$ are $C^1(\R_+ \times \R_+)$ and satisfy conditions
which are typical of Lotka-Volterra type reaction terms for competing
species, i.e. there exist constants $C_f$, $C_f'$, $C_g$, $C_g'>0$
such that for all $u,v\ge0$ (and denoting by $\pa_1,\pa_2$ the
derivatives with respect to the first and second variable)
\begin{equation}\tag{R1}\label{hp R1}
  \begin{split}
    -C_f(1+u+v)\le f(u,v)&\le C_f,\\
    -C_g(1+u+v)\le g(u,v)&\le C_g,\\
    |\partial_1 f(u,v)|,|\partial_2 f(u,v)|&\le C'_f,\\
    |\partial_1 g(u,v)|,|\partial_2 g(u,v)|&\le C'_g.
  \end{split}
\end{equation}
We denote
\begin{equation}
  \label{eq:definition-a}
  A(u,v) \coloneqq u\, B(u,v),
\end{equation}
and assume that
\begin{equation}\tag{D1}\label{hp D1}
  B\in C^1(\R_+ \times \R_+,\R_+).
\end{equation}
We also suppose that there exist $a_0, a_1, a_3>0$ such that for all $ u,v\ge0,$
\begin{equation}\label{def B}
  0 < a_0 \le B(u,v) \le a_1,
\end{equation}
and
\begin{equation}\label{d2 B in Linf}
  \text{$|\partial_2 B(u,v)|\le a_3.$}
\end{equation}
We assume moreover that there exists $a_2>0$ such that for all $ u,v\ge 0,$
\begin{equation}\tag{D2}\label{hp D2}
  0<a_0\le\partial_1A(u,v)\le a_1\qquad\text{ and }\qquad|\partial_2A(u,v)|\le a_2.
\end{equation}\medskip

The motivation for studying such systems comes from the modeling of
the effect of starvation on the movement of individuals belonging to
species in competition. They are sometimes called {\it{starvation
    driven}} cross-diffusion systems. In \cite[Chapter~3]{Brocchieri},
a whole class of {\it{starvation driven}} cross-diffusion systems is
obtained from a microscopic description of the interaction between
individuals. The obtained systems are
\begin{equation*}
  \left\{
    \begin{lgathered}
      \partial_tu-\Delta(d_au^*_a(u,v)+d_bu^*_b(u,v))= f_a(u_a^*(u,v), u^*_b(u,v),v)+f_b(u^*_a(u,v),u^*_b(u,v),v),\\
      \partial_tv- d_v\Delta v =f_v(u^*_a(u,v),u^*_b(u,v),v),
    \end{lgathered}
  \right.
\end{equation*}
over \((0,+\infty) \times \Omega\), with diffusion coefficients
$d_a,d_b,d_v>0$, $d_a\neq d_b$, and  reaction terms
$f_a, f_b, f_v$ of Lotka-Volterra type (for competing species), where
$(u^*_a(u, v), u^*_b(u, v))$ are defined as the unique solution
to the nonlinear system
\begin{equation*}
  \begin{cases}
    u=u^*_a(u,v) + u^*_b(u,v),\\
    \phi(bu_b^*(u,v)+dv)\, u_b^*(u,v)
    -\psi(au_a^*(u,v)+cv)\, u_a^*(u,v)=0,
  \end{cases}
\end{equation*}
where $\phi$ and \(\psi\) are suitable conversion rates, and
$a,b,c,d>0$ are parameters, see \cite[Chapter~3]{Brocchieri} for more
details.  Those systems belong to the class of systems studied in this
paper by setting
\begin{gather*}
  A(u,v) := d_au^*_a(u,v)+d_bu^*_b(u,v), \qquad  vg(u, v) : = f_v(u^*_a(u,v),u^*_b(u,v),v), \\
  u\,f(u,v) := f_a(u_a^*(u,v), u^*_b(u,v),v)+f_b(u^*_a(u,v),u^*_b(u,v),v).
\end{gather*}
\medskip

Our main theorem shows the existence of strong solutions to  system
\eqref{macro step 1}--\eqref{initial data step1} under
Assumption~A.

\begin{thm}\label{thm existence (u,v) + regularity}
  Consider a smooth bounded domain \(\Omega \subset \R^N\) for
  \(N \ge 1\) and suppose that $d_v, B, f,g$ satisfy Assumption~A.
  Consider also nonnegative initial data
  $u_\init\in \big(L^{\infty}\cap H^1\big) (\Omega),$
  $v_\init\in\big(L^{\infty}\cap W^{2,p}(\Omega) \cap
  H^3\big)(\Omega)$ for all $p \in [1, +\infty)$, compatible with the
  Neumann boundary condition \eqref{boundary cond macro step 1}.

  Then, there exists a strong nonnegative global solution $(u,v)$ to system \eqref{macro step 1}--\eqref{initial data step1}, in the sense that
  \begin{enumerate}[label={\roman*)}]
  \item each term in the two equations of \eqref{macro step 1} is locally integrable and the
    equations are satisfied a.e. in $\Omega_T$,
  \item the boundary and initial conditions \eqref{boundary cond
      macro step 1}, \eqref{initial data step1} hold in the
    sense of traces.
  \end{enumerate}
  Moreover, for all $T>0$ and for $i,j=1,\dots, N$, it holds for all $p\in [1,+\infty)$
  \begin{enumerate}[label={\roman*)}]
  \item $u\in L^{\infty}\big([0,T]; L^p(\Omega)\big),$
    $\partial_{x_i} u\in L^{\infty}\big([0,T]; L^2(\Omega)\big),$ $\partial_tu,\partial^2_{x_ix_j} A(u,v)\in L^2(\Omega_T),$
  \item $v\in L^{\infty}(\Omega_T),$
    $\partial_tv, \partial^2_{x_ix_j} v\in L^p(\Omega_T),$
    $ \partial^2_{tt} v, \partial^3_{tx_ix_j} v  \in L^2(\Omega_T)$.
  \end{enumerate}
  Finally, if $N\le 3$, it holds for  $i,j=1,\dots ,N$,
  \begin{equation*}
    u\in L^{\infty}(\Omega_T),\qquad
    \partial_{x_i} u \in L^4(\Omega_T),\qquad \partial_t v, \pa^2_{x_ix_j} v  \in L^{2}([0,T]; L^{\infty}(\Omega)),
  \end{equation*}
  and if $N=1$ and $u_{in} \in  W^{1,p}(\Omega)$ for all $p\in [1,+\infty)$, it holds
  \begin{equation*}
    \partial_{x} u \in L^{p}(\Omega_T).
  \end{equation*}
\end{thm}
\textit{Theorem \ref{thm existence (u,v) + regularity}} is proved in
\textit{Section \ref{sec:existence}}. The main difficulty in
constructing strong solutions to system \eqref{macro step
  1}--\eqref{initial data step1} is the presence of the nonlinear
cross-diffusion term in the equation satisfied by $u$. A key estimate
is formally obtained by using suitable multiplicators when considering
the evolution of \(u\) in \eqref{macro step 1}.  It leads to the
estimates:
$\partial_{x_i} u\in L^{\infty}\big([0,T]; L^2(\Omega)\big),$
$\partial_tu,\partial^2_{x_ix_j} A(u,v)\in L^2(\Omega_T),$ for
$i,j=1,\dots ,N$.  We present in this paper an original way of finding
a suitable regularization preserving these estimates (see
\textit{Subsection \ref{subsec:intro reg syst}}). This is achieved by
first proving the existence of solutions to a regularized system in a
\textit{non-divergence} form, using a convenient change of variable
(see \textit{Subsections \ref{subsec:exist reg syst nondiv},
  \ref{subsec: exist reg syst div}}). Then, we show uniform (with
respect to the regularization parameter) a priori bounds, which allow
to pass to the limit when the regularization parameter tends to $0$
(see \textit{Subsection \ref{subsec: end proof main thm}}).

\medskip
We conclude this introduction with the statement of two
stability/uniqueness results. If the space dimension is $N \le 2$, the
regularity of the strong solutions obtained in \textit{Theorem
  \ref{thm existence (u,v) + regularity}}, is sufficient to prove
stability in a strong norm (see \textit{Theorem \ref{thm uniqueness}})
under the following slightly stricter assumption.  \smallskip

\noindent\textbf{Assumption B}. The diffusivity function $A$
is such that
\begin{equation*}
  \partial_{ij}^2 A \text{ is bounded for  } i,j =1,2 .
\end{equation*}

This regularity is also sufficient to ensure stability in a weak norm if $N \le 3$, without the extra Assumption~B. More precisely, we prove in \textit{Theorem \ref{nuq}} a stability result in $(H^1)'(\Omega)$, denoted as the dual space of $H^1(\Omega)$,  with the norm
\begin{equation}\label{norm dual H1}
  \|w\|_{(H^1)'(\Omega)}^2 := (w)_{\Omega}^2+\|\nabla \phi \|_{L^2(\Omega)}^2, \qquad w\in L^2(\Omega),
\end{equation}
where $(w)_{\Omega}$ is the average value of $w$ in $\Omega$ and $\phi$
is the unique solution of the Neumann problem
\begin{equation*}
  - \Delta \phi = w -   (w)_{\Omega}, \quad {\hbox{in}}\,\Omega,\qquad\quad \nabla \phi \cdot \sigma = 0 \quad {\hbox{on}}\,\pa\Omega, \qquad\quad (\phi)_{\Omega} =0.
\end{equation*}
\smallskip

\begin{thm}\label{thm uniqueness}
  Let $N\le2$ and $\Omega$ be a smooth bounded open set of $\R^N$. We
  suppose that Assumptions~A and B hold, and we take two solutions
  $(u_i,v_i),\,i=1,2$ of system \eqref{macro step 1}, \eqref{boundary
    cond macro step 1}, given by \textit{Theorem~\ref{thm existence
      (u,v) + regularity}}, corresponding to the nonnegative initial
  data $(u_{i,\,\init},v_{i,\,\init}),\,i=1,2$ with
  $u_{i,\init}\in (L^{\infty} \cap H^1)(\Omega)$,
  $v_{i,\init}\in H^3(\Omega).$

  Then, there exists a constant $C_{stab}>0$, depending on
  $ \Omega,$ $T,$ $a_0,$ $a_1,$ $a_2,$ $a_3,$ $d_v,$ $C_f,$ $C_f'$, $C_g,$ $C_g'$,
  $\|\pa^2_{ij} A\|_{\infty}$,
 and on $\|u_2\|_{L^{\infty}(\Omega_T)}$,
  $\|\nabla u_2\|_{L^{4}(\Omega_T)}$,
  $\|v_2\|_{L^{\infty}(\Omega_T)}$,
  $\|\nabla v_2\|_{L^2([0,T]; L^{\infty}(\Omega))}$, such that
  \begin{multline}\label{uniqueness-ineq-prop}
    \| u_1 - u_2\|_{L^\infty([0,T]; L^2(\Omega))\,\cap \,L^2([0,T]; H^1(\Omega))}^2 + \|v_1 - v_2\|_{L^\infty([0,T]; L^2(\Omega))\,\cap\, L^2([0,T]; H^1(\Omega))}^2\\
    \le C_{stab}\big(\|u_{1,\init}-u_{2,\init}\|^2_{L^2(\Omega)}+\|v_{1,\init}-v_{2,\init}\|^2_{L^2(\Omega)}\big).
  \end{multline}
  Finally, if $u_{1,\init}=u_{2,\init}$ and $v_{1,\init}=v_{2,\init}\,$ for a.e. $x\in \Omega,$ then
  \[
    u_1(t,x)=u_2(t,x)\qquad\text{ and }\qquad v_1(t,x)=v_2(t,x),\quad\text{ a.e. $(t,x)\in \Omega_T$},
  \]
  so that uniqueness holds for system \eqref{macro step 1}, \eqref{boundary cond macro step 1}.
\end{thm}

\begin{thm}\label{nuq}
  Let $N\le 3$ and $\Omega$ be a smooth bounded open set of $\R^N$. We
  suppose that Assumption~A holds, and we take two solutions
  $(u_i,v_i),\,i=1,2$ of system \eqref{macro step 1}, \eqref{boundary
    cond macro step 1}, given by \textit{Theorem~\ref{thm existence
      (u,v) + regularity}}, corresponding to the nonnegative initial
  data $(u_{i,\,\init},v_{i,\,\init}),\,i=1,2$ with
  $u_{i,\init}\in (L^{\infty} \cap H^1)(\Omega)$,
  $v_{i,\init}\in \left(W^{2,p} \cap H^3\right)(\Omega)$ for all
  $p\in [1,+\infty)$.

  Then, there exists a constant $C_{stab}'>0$, depending on $ \Omega,$
  $T,$ $a_0,$ $a_1,$ $a_2,$ $a_3,$ $d_v,$ $C_f,$ $C'_f$, $C_g,$
  $C_g'$, and on $\|u_1\|_{L^{\infty}(\Omega_T)}$,
  $\|u_2\|_{L^{\infty}(\Omega_T)}$, $\|v_1\|_{L^{\infty}(\Omega_T)}$,
  $\|v_2\|_{L^{\infty}(\Omega_T)}$, such that
  \begin{multline}\label{uniqueness-ineq-prop}
    \| u_1 - u_2\|_{L^\infty([0,T]; (H^1)'(\Omega))}^2 + \|v_1 - v_2\|_{L^\infty([0,T]; (H^1)'(\Omega))}^2
    \\
    \le C_{stab}'\big(\|u_{1,\init}-u_{2,\init}\|^2_{(H^1)'(\Omega)}+\|v_{1,\init}-v_{2,\init}\|^2_{(H^1)'(\Omega)}\big).
  \end{multline}

  Finally, if $u_{1,\init}=u_{2,\init}$ and $v_{1,\init}=v_{2,\init}\,$ for a.e. $x\in \Omega,$ then
  \[
    u_1(t,x)=u_2(t,x)\qquad\text{ and }\qquad v_1(t,x)=v_2(t,x),\quad\text{ a.e. $(t,x)\in \Omega_T$},
  \]
  so that uniqueness holds for system \eqref{macro step 1}, \eqref{boundary cond macro step 1}.
\end{thm}

\begin{remark}
  The stability results only depend on the dimension through the
  assumptions on the solutions (as obtained in \textit{Theorem
    \ref{thm existence (u,v) + regularity}}) and can be rephrased as
  conditional stability results in all dimensions.

  For the weak stability result, the stability constant depends on
  both solutions but the dependency on the first solution
  \((u_1,v_1)\) can be measured in a weaker norm, see the remark at
  the end of the proof of \textit{Theorem \ref{nuq}}.
\end{remark}

For the modeling of biological populations in starvation driven situations, we refer to \cite{ChoKim}, \cite{KIM2014} and \cite{Brocchieri}. We refer to \cite{SKT} for one of the first models in population dynamics involving cross-diffusion terms. Existence, uniqueness and regularity results for triangular cross-diffusion systems can be found in \cite{Lou1998}, \cite{Yamada},  \cite{Conf-Des-Sores},
 \cite{Desvillettes}, \cite{DesTres} (note here that by triangular, we mean that cross-diffusion terms appear only in one of the equations of the system).  Finally, for results on the stability of equilibria for cross-diffusion models, we refer to \cite{IMN},  \cite{Lombardo2013}, \cite{Lombardo2012}, \cite{Lombardo2008}, \cite{Capone}, \cite{Des-Sores}, and the references therein.
\medskip

The rest of the paper is organised as follows: \textit{Section \ref{sec:existence}} aims to prove the existence result, stated in \textit{Theorem \ref{thm existence (u,v) + regularity}}. More precisely, in \textit{Subsection \ref{subsec:intro reg syst}}, we introduce a regularized system in \textit{non-divergence} form and in \textit{Subsection \ref{subsec:exist reg syst nondiv}}, we prove the existence of a strong solution to this regularized  system. Then, we explain in \textit{Subsection \ref{subsec: exist reg syst div}  } how it enables to obtain existence for an equivalent regularized system in \textit{divergence} form.
In \textit{Subsection \ref{subsec: end proof main thm}}, we conclude the proof of  \textit{Theorem \ref{thm existence (u,v) + regularity}} by removing the regularization. The stability statements are then proven in \textit{Section \ref{sec:uniqueness}}. A classical result for parabolic equations is finally recalled in \textit{Appendix \ref{app-sect:esistence}}.

\section{Proof of the main Theorem~\ref{thm existence (u,v) + regularity}}\label{sec:existence}

\subsection{Introduction of a regularized  system in \textit{non-divergence} form}\label{subsec:intro reg syst}

We first introduce a parabolic system which is formally equivalent to
\eqref{macro step 1}--\eqref{initial data step1}. For
this, we consider the equation satisfied by  $a := A(u,v)$.

By assumption \eqref{hp D2}, we can define the reciprocal $U$ of
$A$ with respect to the first variable, that is, for a given $v\ge0$,
\begin{equation}\label{c of v}
a= A(u,v)\qquad\Longleftrightarrow\qquad u= U(a,v).
\end{equation}
Using this change of variable, we can rewrite \eqref{macro step
  1}--\eqref{initial data step1} as follows:
\begin{equation*}
  \begin{cases}
    \partial_t a=\mu(a,v)\Delta a+a \,s(a,v,\partial_tv),\qquad\qquad &\text{in }\Omega_T,\\
    \partial_t v=d_v\Delta v+vg(U(a,v), v),&\text{in }\Omega_T,\\
    \nabla a\cdot\sigma=\nabla v\cdot\sigma=0,&\text{on }(0,T)\times\pa\Omega,\\
  \end{cases}
\end{equation*}
with $a(0, \cdot) = a_{\init}\coloneqq A(u_{\init},v_{\init})$, $v(0,\cdot) = v_{in}$,
\begin{equation}\label{def mu}
  \mu(a,v)\coloneqq \partial_1 A(U(a,v),v),
\end{equation}
and for all $a > 0$
\begin{equation*}
  s\big(a,v,\partial_t v\big)\coloneqq
  \f{U(a,v)}{a}\left[f\big(U(a,v),v\big)\partial_1 A(U(a,v),v)
    +\partial_2B\big(U(a,v),v\big)\partial_tv\right].
\end{equation*}

By \eqref{hp D1}, \eqref{hp D2} we observe that from
\eqref{def mu}, $\mu\in C(\R_+ \times \R_+)$ and
\begin{equation}\label{hp mu}
     \text{for all} \,\,\, a,v\ge 0,\qquad 0<a_0\le\mu(a,v)\le a_1,
\end{equation}
and thanks to estimate \eqref{def B}, we get for all $a>0$ and $v\ge0$
\begin{equation}\label{property U/a}
  0<\f 1{a_1}\le \f{U(a,v)}{a}\le \f 1{a_0}.
\end{equation}
Moreover by \eqref{hp D1}, \eqref{hp D2}, the implicit
function theorem guarantees the $C^1$ character of $U$ and, for all
$a,v\ge0$, the estimates
\begin{equation}\label{pa1 U}
\f1{a_1}\le\partial_1 U(a,v)=\Big(\partial_1 A\big(U(a,v),v\big)\Big)^{-1}=\f 1{\mu(a,v)}\le \f 1{a_0},
\end{equation}
using \eqref{def mu}, \eqref{hp mu}, and thanks to (\ref{hp D2}),
\begin{equation}\label{pa2 U}
  -  \f{a_2}{a_0} \le \partial_2 U(a,v)=
 -\f {\partial_2 A\big(U(a,v),v\big)}{\mu(a,v)} \le \f{a_2}{a_0}.
\end{equation}
We now introduce the following truncated-regularized system in
\textit{non-divergence} form
 for any $\e,M>0$, 
\begin{equation}\label{system regular eps (a, v)}
  \begin{cases}
    \partial_t a_{\e,M}=\mu(a_{\e,M},v_{\e,M})\Delta
    a_{\e,M}+a_{\e,M}s_{M}(a_{\e,M}, v_{\e,M},\partial_t v_{\e,M}),
    &\text{in }\Omega_T,\\
    \partial_t v_{\e,M}=d_v\Delta
    v_{\e,M}+v_{\e,M}g_{\e,M}(U(a_{\e,M},v_{\e,M}), v_{\e,M}),
    &\text{in }\Omega_T,\\
    \nabla a_{\e,M}\cdot\sigma=\nabla v_{\e,M}\cdot\sigma=0,&\text{on }(0,T)\times\pa\Omega,\\
  \end{cases}
\end{equation}
with
\begin{equation}\label{def s_M}
  s_M\big(a,v,\partial_t v\big)
  \coloneqq
  \f{U(a,v)}{a}\left[f\big(\min\{U(a,v),M\} ,v\big)\partial_1 A(U(a,v),v)
    +\partial_2B\big(U(a,v),v\big)\partial_tv\right],
\end{equation}
and where we define, for a.e. $(t,x)\in \R^{N+1}$ and for all $\e>0$,
\begin{equation}\label{def g M eps}
  g_{\e,M}\big(u, v\big)
  \coloneqq g\big(\min\{ u, M\}, \min\{v,M\} \big)\ast_{t,x} \varphi_\e .
\end{equation}
Note that, slightly abusing notations, $u$ and $v$ in \eqref{def g M
  eps} are identified to their respective extension defined on
$\R^{N+1}$, by continuity in time outside $(0,T)$ and by zero in space
outside $\Omega$. Finally, $\ast_{t,x}$ stands for the convolution
operation in time and space variables and $(\varphi_\e)_{\e>0}$ is a
family of standard mollifiers on $\R^{N+1}$.

 Moreover, the system \eqref{system
  regular eps (a, v)}--\eqref{def g M eps} is endowed with the
initial data
\begin{equation}\label{initial data general ae ve}
  \begin{split}
    a_{\e,M}(0,x)
    &=a_{\init,\e}(x)=A\left(\big( u_{\init}\ast_x\psi_\e\big)(x)
      ,
      ( v_{\init}\ast_x\psi_\e)(x)    \right),\qquad\quad x\in \Omega,\\
    v_{\e,M}(0,x)
    &=v_{\init,\e}(x)=( v_{\init}\ast_x\psi_\e)(x), \hspace{3.77cm} x\in \Omega ,
  \end{split}
\end{equation}
where again, $u_{\init}, v_{\init}$ are extended by zero outside $\Omega$, and $(\psi_\e)_{\e>0}$ is a family of standard
mollifiers on $\R^{N}$.

It is worth noticing that the regularization and truncation only
affect the reaction part in \eqref{system regular eps (a, v)} and the
initial conditions \eqref{initial data general ae ve}. Note also that
we truncate the function $f$ only with respect to $u$, while no
truncation with respect to $v$ is needed.

\subsection{Existence for the regularized system in \textit{non-divergence} form}\label{subsec:exist reg syst nondiv}

In this subsection, we prove existence to  system
\eqref{system regular eps (a, v)}--\eqref{initial data
  general ae ve}. We first state the following existence result.

\begin{prop}\label{prop delta limit}
  Let $N\ge 1$ and $\Omega$ be a smooth bounded open
  subset of $\R^N$. We suppose that the parameters $d_v, B, f,g$
  satisfy Assumption~A.  We consider nonnegative initial data
  $(a_{\init},v_{\init})\in (L^{\infty} \cap H^1) (\Omega) \times
  L^{\infty}(\Omega).$

  Then, for any fixed $\e,M>0,$ there exists
  a nonnegative strong (in the same sense as in \textit{Theorem~\ref{thm
    existence (u,v) + regularity}}) solution $(a_{\e,M},v_{\e,M})$ to
  \eqref{system regular eps (a, v)}--\eqref{initial data
    general ae ve}, such that for all $i,j=1,\dots,N$,
  \begin{enumerate}[label={\roman*)}]
  \item $a_{\e,M}\in L^{\infty}(\Omega_T),$
    $\,\partial_t a_{\e,M},\partial^2_{x_ix_j} a_{\e,M}\in
    L^2(\Omega_T),\partial_{x_i} a_{\e,M}\in L^4(\Omega_T)$,
  \item
    $v_{\e,M},\,\partial_{x_i}
    v_{\e,M}, \, \partial_t v_{\e,M}  \in {L^{\infty}(\Omega_T)}
    ,\, \partial^2_{x_ix_j} v_{\e,M} \in
    L^p(\Omega_T)$ for all $p \in [1, \infty)$.
  \end{enumerate}
\end{prop}
\medskip

We now prove \textit{Proposition~\ref{prop delta limit}}.

\begin{proof}
We want to use a fixed point argument. In order to be able to use
Schauder's theorem, we need first to introduce yet another
regularization.

Let $\e,M>0$ be fixed. We introduce the approximating system below for
all $\delta>0$ (we only indicate the dependence of the unknowns $a,v$
with respect to $\delta$ since $\e$ and $M$ are fixed),
\begin{equation}\label{system Schauder}
  \begin{cases}
    \partial_t a_{\delta}=\big( \mu(a_{\delta},v_{\delta})\ast_x\psi_\delta\big)\Delta a_{\delta}+a_{\delta}s_{M}(a_{\delta}, v_{\delta},\partial_t v_\delta),\qquad&\text{in }\Omega_T,\\
    \partial_t v_{\delta}=d_v\Delta v_{\delta}+v_{\delta}\,g_{\e,M}(U(a_{\delta},v_{\delta}), v_{\delta}),&\text{in }\Omega_T,\\
    \nabla a_{\delta}\cdot\sigma=\nabla v_{\delta}\cdot\sigma=0,&\text{on }(0,T)\times\pa\Omega,
  \end{cases}
\end{equation}
where, once again slightly abusing notations, $\mu(a_\delta,v_\delta)$ is identified to its extension by zero defined on $[0,T]\times\R^N$, while $(\psi_\delta)_{\delta}$ is a standard mollifier on
$\R^N$. Finally, $\mu, s_M,\,g_{\e,M}$ are defined as in \eqref{def
  mu}, \eqref{def s_M}, \eqref{def g M eps} respectively,
and the nonnegative initial data  are defined as in
\eqref{initial data general ae ve}, i.e.
\begin{equation}\label{initial data (ae,ve)}
  a_{\delta}(0,x) :=a_{\init,\e}(x)\ge 0,\qquad
  v_{\delta}(0,x) :=v_{\init,\e}(x)\ge 0.
\end{equation}

The existence to system \eqref{system Schauder}, \eqref{initial data
  (ae,ve)} is obtained by applying Schauder's fixed point theorem
\cite{Evans} on the Banach space
\begin{equation*}
  E\coloneqq L^{\infty}\big([0,T]; L^2(\Omega)\big),
\end{equation*}
and with the map
\begin{equation}\label{def map Phi}
  \Phi_{\delta}: (a_\delta,v_\delta)\in E^2\mapsto (\bar a_\delta,\bar v_\delta),
\end{equation}
where $\bar v_\delta$ satisfies
\begin{equation}\label{eq bar ve}
  \begin{cases}
    \partial_t\bar v_\delta=d_v\Delta \bar v_\delta+\bar v_\delta g_{\e,M}\big(U(a_\delta,v_\delta),v_\delta\big),\qquad&\text{in }\Omega_T,\\
    \nabla \bar v_\delta\cdot\sigma=0,\qquad &\text{on }(0,T)\times\pa\Omega,\\
    \bar v_\delta(0,x)=v_{\init,\e}, \qquad&\text{on } \Omega,
  \end{cases}
\end{equation}
and $\bar a_{\delta}$ solves
\begin{equation}\label{eq bar ae}
  \begin{cases}
    \partial_t\bar  a_{\delta}=\big(\mu(a_{\delta},\bar v_{\delta})\ast_x\psi_\delta\big)\Delta \bar a_{\delta}+\bar a_{\delta}s_{M}(a_{\delta}, \bar v_{\delta}, \partial_t\bar v_\delta),&\text{in }\Omega_T,\\
    \nabla \bar a_\delta\cdot\sigma=0, &\text{on }(0,T)\times\pa\Omega,\\
    \bar a_\delta(0,x)=a_{\init,\e},&\text{on } \Omega.
  \end{cases}
\end{equation}

Note that \eqref{eq bar ve} and \eqref{eq bar ae} are
linear parabolic problems for which the existence of a unique
nonnegative strong solution is classical (see
\textit{Proposition~\ref{prop 1 eq in b}} of the Appendix).  Then, we show that the map
\eqref{def map Phi} satisfies the assumptions of Schauder's
fixed point theorem.
\medskip

Indeed, thanks to the maximum principle for the heat equation, we first see that since
$$ g_{\e,M}\big(U(a_\delta,v_\delta),v_\delta\big) \le C_g, $$
we get the estimate
\begin{equation}\label{ineq Schauder 1}
  \|\bar v_{\delta}\|_{L^{\infty}(\Omega_T)}\le e^{TC_g}\|v_{\init,\e}\|_{L^{\infty}(\Omega)} .
\end{equation}
Then, we take the derivative in time of (\ref{eq bar ve}):
$$ (\pa_t - d_v \Delta) (\pa_t \bar v_{\delta}) = \pa_t \bar v_{\delta} \,  \, g\big(\min\{ U(a_{\delta}, v_{\delta}), M\}, \min\{v_{\delta}, M\}\big)\ast_{t,x}  \varphi_\e $$
$$+ \,\bar v_{\delta}  \, g\big(\min\{ U(a_{\delta}, v_{\delta}) , M\}, \min\{v_{\delta}, M\} \big)\ast_{t,x} \pa_t \varphi_\e .$$
Observing that
$$ \| g\big(\min\{ U(a_{\delta}, v_{\delta}), M\}, \min\{v_{\delta}, M\}\big)\ast_{t,x}  \varphi_\e \|_{L^{\infty}(\Omega)} $$
$$ + \, \| g\big(\min\{ U(a_{\delta}, v_{\delta}), M\}, \min\{v_{\delta}, M\} \big)\ast_{t,x} \pa_t \varphi_\e  \|_{L^{\infty}(\Omega)}  \le C(\e, M, T), $$
where $C(\e, M, T)$ is a generic constant depending on $\e, M, T$ but not $\delta$, and using estimate (\ref{ineq Schauder 1}), we can now use
the maximum principle for the heat equation, and get
\begin{equation}\label{dtch}
  \| \pa_t \bar v_{\delta} \|_{L^{\infty}(\Omega_T)}\le C(\e, M, T) , \qquad
\| s_M(a_{\delta}, \bar v_{\delta}, \pa_t \bar v_{\delta} ) \|_{L^{\infty}(\Omega_T)}\le C(\e, M, T),
\end{equation}
with $s_M$ is defined in \eqref{def s_M} and using \eqref{hp D2}, \eqref{hp R1}, \eqref{d2 B in Linf}, \eqref{property U/a}. Moreover, the maximal regularity implies
\begin{equation}\label{ineq Schauder 2}
 \|\partial^2_{x_ix_j}\bar v_\delta\|_{L^p(\Omega_T)} \le  C(\e,M,T),
\end{equation}
for all $i,j=1,\dots,N$, $ p\in [1, +\infty)$.  Then $\bar a_\delta$
satisfies, thanks to \textit{Proposition~\ref{prop 1 eq in b}} (and an
interpolation between $L^{\infty}(\Omega_T)$ and
$L^2([0,T]; H^2(\Omega))$) for all $i,j=1,\dots,N$,
\begin{equation}\label{ineq Schauder 3}
 \|\bar a_\delta\|_{L^{\infty}(\Omega_T)} +
  \|\partial_t\bar a_\delta\|_{L^2(\Omega_T)}+\|\partial^2_{x_ix_j}\bar a_\delta\|_{L^2(\Omega_T)}+\|\partial_{x_i}\bar a_\delta\|_{L^4(\Omega_T)}\le  C(\e,M,T).
\end{equation}

Thanks to estimates \eqref{ineq Schauder 1}--\eqref{ineq Schauder 3},
the map $\Phi: E^2\to E^2$ is compact, using the Rellich-Kondrakov
theorem \cite{Brezis}. Moreover, it satisfies the inclusion
$$\Phi(B_Q\times B_Q)\subset B_Q\times B_Q,$$
with
$B_Q\coloneqq B_{L^{\infty}(\Omega_T)}(0,Q)_+ :=\{ w \in E \text{
  s.t. } 0 \le w \le Q \}$ and
$$Q\coloneqq \max \left( e^{T\,C_g}\, \|
  v_{\init,\e}\|^2_{L^{\infty}(\Omega)}, C(\e, M, T) \right)>0.$$
Finally, it is possible to show that $\Phi: E^2\to E^2$ is continuous
(see \cite[\textit{Subsection 3.2.3}]{Brocchieri} for details).
\medskip

Therefore, using Schauder's fixed point theorem, there exists at least
one strong solution $(a_\delta \ge 0, v_\delta \ge 0)$ to
\eqref{system Schauder}, \eqref{initial data (ae,ve)},
satisfying estimates \eqref{ineq Schauder 1}--\eqref{ineq
  Schauder 3} (with $(a_\delta,v_\delta)$ replacing
$({\bar a}_\delta, {\bar v}_\delta)$).  \medskip
\medskip

Using the second equation of \eqref{system Schauder} and the
properties of the heat equation, we also observe that
\begin{equation}\label{neq1}
 \|\partial_{x_i}\bar v_\delta\|_{L^{\infty}(\Omega_T)} \le  C(\e,M,T), \qquad \mbox{for $i=1,..,N$. }
\end{equation}

We now use this solution $(a_\delta, v_\delta)$ and let
$\delta \to 0$.  Estimates \eqref{ineq Schauder
  1}--\eqref{ineq Schauder 3} ensure that we can extract
subsequences (still denoted by $(a_\delta, v_\delta)$) such that for
some $a,v\in L^{\infty}(\Omega_T),$
\begin{equation}\label{conv a.e. adelta vdelta}
  a_\delta\rightarrow a,\,\, v_\delta\rightarrow v\quad\text{ a.e.\ in $\Omega_T$ and strongly in $L^p(\Omega_T),$ for all $ p\in [1, \infty)$,}
\end{equation}
\begin{equation}\label{converg a delta}
\partial_t a_\delta\rightharpoonup\partial_t a,\,\,\Delta a_\delta\rightharpoonup\Delta a,\,\,\nabla a_\delta\rightharpoonup\nabla a,\quad\text{weakly in } L^2(\Omega_T),
\end{equation}
and 
\begin{equation}\label{converg v delta}
\partial_t v_\delta\rightharpoonup\partial_t v,\,\,\Delta v_\delta\rightharpoonup\Delta v,\,\,\nabla v_\delta\rightharpoonup\nabla v,\quad\text{weakly in } L^p(\Omega_T),\text{ for all $ p\in [1, \infty)$}.
\end{equation}
Recalling  estimate \eqref{hp mu}, we end up with
\begin{equation}\label{L1weak limit diff term a}
\big(\mu(a_\delta,v_{\delta})\ast_x\psi_\delta\big)\Delta a_{\delta}\rightharpoonup\mu(a,v)\Delta a,\quad\text{weakly in $L^1(\Omega_T),\,\,\,$ as $\delta\rightarrow 0.$}
\end{equation}
We now take the weak limit in the r.h.s. of the first equation in
\eqref{system Schauder}. The convergences \eqref{conv
  a.e. adelta vdelta}, \eqref{converg v delta} and the
$\delta-$uniform bounds \eqref{ineq Schauder
  1}--\eqref{neq1} ensure that
$s_M(a_\delta,v_\delta,\partial_t v_\delta)$ converges to $s_M(a,v,\partial_t v)$, weakly in $L^p(\Omega_T)$ for all $p\in [1, \infty)$, so that $(a,v)$ satisfy the first equation of system \eqref{system Schauder}.

The convergence in the second equation in system \eqref{system
  Schauder} is obtained in a similar way.  We also pass to the
limit in the boundary condition of \eqref{system Schauder},
using \eqref{conv a.e. adelta vdelta}--\eqref{converg v delta} and
the continuity of the trace
operator.
\medskip

 Finally, using the weak lower
semi-continuity property of the $L^p(\Omega_T)$ norm for
$p\in (1,+\infty]$ and estimates \eqref{ineq Schauder
  1}--\eqref{neq1}, we conclude that
$(a_{\e,M},v_{\e,M})$ is a strong solution to the system
\eqref{system regular eps (a, v)}--\eqref{initial data
  general ae ve}, satisfying the bounds i) and ii) announced in \textit{Proposition \ref{prop delta limit}}.
\end{proof}

\subsection{Existence for the regularized system in \textit{divergence} form}\label{subsec: exist reg syst div}

Hereafter, we restore the $\e,M-$dependency in the notation, so that
we refer to the a.e.\ limit of $a_\delta,v_\delta$ as
$a_{\e,M},v_{\e,M},$ respectively.
\medskip

In this subsection, we deduce from \textit{Proposition \ref{prop delta
  limit}} the existence to the following regularized system in
\textit{divergence} form, satisfied by the unknowns
$(u_{\e,M}, v_{\e,M})$
\begin{equation}\label{system eps (u,v)}
\begin{cases}
\partial_t u_{\e,M}=\Delta \big(A(u_{\e,M},v_{\e,M})\big)+u_{\e,M}f_M(u_{\e,M}, v_{\e,M}),\qquad &\text{on }\Omega_T,\\
\partial_t v_{\e,M}=d_v\Delta v_{\e,M}+v_{\e,M}g_{\e,M}(u_{\e,M}, v_{\e,M}),&\text{on }\Omega_T,\\
\nabla \big(A(u_{\e,M},v_{\e,M})\big)\cdot\sigma=\nabla v_{\e,M}\cdot\sigma=0,&\text{on }\,(0,T)\times\pa\Omega,
\end{cases}
\end{equation}
with $f_M(u,v) \coloneqq f(\min\{u,M\},v)$ (and, as previously,
$g_{\e,M}\big(u, v\big)\coloneqq g\big( \min\{ u,
M\}, \min\{v, M \} \big)\ast_{t,x} \varphi_\e$). We complete the
system with the initial conditions
\begin{equation}\label{def initial data ue ve}
\begin{split}
u_{\e,M}(0,x)&=u_{\init,\e}(x) :=\big(u_{\init}\ast_x\psi_\e\big)(x),\qquad \forall\,x\in\Omega,\\
v_{\e,M}(0,x)&=v_{\init,\e}(x) :=\big(v_{\init}\ast_x\psi_\e\big)(x), \qquad\, \forall\,x\in\Omega.
\end{split}
\end{equation}

The existence result is stated in the following corollary.

\begin{cor}\label{prop existence regul-trunc (ue, ve) system}
  Let $N\ge 1$ and $\Omega$ be a smooth bounded open
  subset of $\R^N$. We suppose that the parameters $d_v, B, f,g$
  satisfy Assumption~A.  We consider the initial data
  $u_\init\in (L^{\infty}\cap H^1)(\Omega),$
  $v_\init\in L^{\infty}(\Omega)$ compatible with Neumann boundary
  condition.

  Then, for any fixed $\e,M>0$ there exists a nonnegative strong (in
  the sense of Theorem~\ref{thm existence (u,v) + regularity})
  solution $(u_{\e,M},v_{\e,M})$ to system \eqref{system eps
    (u,v)}, \eqref{def initial data ue ve}, such that for all $i,j=1,\dots,N, $
  \begin{enumerate}[label={\roman*)}]
  \item $u_{\e,M}\in L^{\infty}(\Omega_T),$ $\,\partial_t u_{\e,M},\partial^2_{x_ix_j} A(u_{\e,M}, v_{\e, M}) \in L^2(\Omega_T),\partial_{x_i} u_{\e,M}\in L^4(\Omega_T)$,
  \item $v_{\e,M},\,\partial_{t} v_{\e,M},\,\partial_{x_i} v_{\e,M} \in L^\infty(\Omega_T)$, $\partial^2_{x_ix_j} v_{\e,M} \in L^p(\Omega_T)$ for all $p \in [1, \infty)$.
  \end{enumerate}
\end{cor}

\begin{proof}
  For any $\e,M$ fixed, let \((a_{\e,M},v_{\e,M})\) be a solution to
  \eqref{system regular eps (a, v)}--\eqref{initial data
    general ae ve}, given by \textit{Proposition~\ref{prop delta
      limit}}. Then, by a density argument and recalling the $C^1$
  character of $U$ in \eqref{c of v}, we can prove that
  $u_{\e,M}=U(a_{\e,M}, v_{\e,M})$ admits a weak time derivative,
  given by
  \begin{equation}\label{weak dt u}
    \pa_t u_{\e,M}=	\partial_1 U(a_{\e,M},v_{\e,M}) \partial_t a_{\e,M}
    + \partial_2 U(a_{\e,M},v_{\e,M}) \partial_t v_{\e,M},
  \end{equation}
  belonging to $L^2(\Omega_T)$, and a weak space derivative
  \begin{equation}\label{weak dx u}
    \partial_{x_i}u_{\e,M}= \partial_1 U(a_{\e,M},v_{\e,M}) \partial_{x_i} a_{\e,M}
    + \partial_2 U(a_{\e,M},v_{\e,M}) \partial_{x_i} v_{\e,M},\,\,\forall i=1,\dots N,
  \end{equation}
  belonging to $L^4(\Omega_T)$. Note also that
  $\pa^2_{x_i x_j} A(u_{\e,M}, v_{\e,M}) = \pa^2_{x_i x_j}
  a_{\e,M}$. Then, using \eqref{weak dt u} and the identity
  $\partial_1 U(a,v)\partial_1 A\big(U(a,v),v\big)=1$, we see that
  $u_{\e,M}$ satisfies the first equation of \eqref{system eps
    (u,v)} (a.e.\ in $\Omega_T$). Moreover, the trace of
  $\nabla a_{\e,M}$ on $[0,T]\times \pa\Omega$ is the trace of
  $\nabla \big(A(u_{\e,M}, v_{\e,M})\big)$, so that the first Neumann
  boundary condition in \eqref{system eps (u,v)} is
  satisfied. The same holds for the initial conditions. Finally, the
  equation, boundary condition and initial condition related to
  $v_{\e,M}$ are identical in the system satisfied by
  $(a_{\e,M},v_{\e,M})$ and in the system satisfied by
  $(u_{\e,M},v_{\e,M})$ (when $U(a_{\e,M}, v_{\e,M})$ is replaced by
  $u_{\e,M}$ in the equations). Therefore, $(u_{\e,M}, v_{\e,M})$ is a
  nonnegative strong solution to system \eqref{system eps (u,v)},
  \eqref{def initial data ue ve}. The bounds satisfied by $u_{\e,M}$, $v_{\e,M}$ in \textit{Corollary \ref{prop existence regul-trunc (ue, ve) system}} are a direct consequence of the bounds satisfied by $a_{\e,M}$, $v_{\e,M}$ in \textit{Proposition \ref{prop delta limit}}.

\end{proof}

\subsection{Concluding the existence result}\label{subsec: end proof main thm}

By the previous subsections, we now have at our disposal solutions to
an approximated problem. Most of the bounds obtained for those
solutions are however depending on the truncation-regularization
parameters $\e$ and $M$. Before letting $\e$ tend to $0$ and $M$ tend
to $\infty$, we need to establish bounds (still for the solutions to
the considered approximated problem) which do not depend on $\e$ and
$M$. This is the object of the lemma below.

\begin{lem}\label{lll}
  We consider the solution $(u_{\e,M},v_{\e,M})$ to system
  \eqref{system eps (u,v)}, \eqref{def initial data ue ve}, given by
  Corollary~\ref{prop existence regul-trunc (ue, ve) system}, and we
  assume that $u_{\init}\in (L^\infty\cap H^1)(\Omega)$ and
  $v_{\init}\in (L^{\infty}\cap W^{2,p} \cap H^3)(\Omega)$ for all $p\in
  [1,+\infty)$.

  Then, $(u_{\e,M},v_{\e,M})$ satisfies the following $\e,M-$uniform a priori estimates, for all $p\in [1,+\infty)$
  \begin{equation}\label{unif eps Linf ve and L4 ue}
    \|v_{\e,M}\|_{L^{\infty}(\Omega_T)}\le C(T), \qquad
    \|u_{\e,M}\|_{L^{\infty}([0,T]; L^p(\Omega))}\le C(T),
  \end{equation}
  and for all $i, j = 1, \dots, N$ 
  \begin{equation}\label{unif eps max reg ve}
    \|\pa_t v_{\e,M}\|_{L^p(\Omega_T)}+\|\pa^2_{x_ix_j}v_{\e,M}\|_{L^p(\Omega_T)}
\le  C(T),
  \end{equation}
  and
  \begin{align}\label{unif eps max regularity ae}
      \|\pa_tu_{\e,M}\|_{L^2(\Omega_T)}&+\|\pa_{x_i} u_{\e,M}
      \|_{L^{\infty}([0,T]; L^2(\Omega))}
\notag\\
      &+\|\pa^2_{x_ix_j} (A(u_{\e,M},v_{\e,M}))\|_{L^2(\Omega_T)}\le  C(T).
  \end{align}
  Moreover, for all $i,j=1,\dots N$ and for any $\eta>0$
  \begin{equation}\label{nez}
    \begin{split}
      \|\pa^3_{tx_ix_j} v_{\e,M}& \|_{L^2(\Omega_T)}  + \|\pa^2_{tt} v_{\e,M} \|_{L^2(\Omega_T)}+\|\pa^2_{t\, x_i}v_{\e,M}\|_{L^{4-\eta}(\Omega_T)}\le C(T).
    \end{split}
  \end{equation}
  Finally, if $N\le 3$,
  \begin{align}\label{L2Linf bound dt v lap v}
    \|\pa_t v_{\e,M}\| _{L^{2}([0,T]; L^{\infty}(\Omega))}&+\| \pa^2_{x_ix_j} v_{\e,M}\| _{L^{2}([0,T]; L^{\infty}(\Omega))} 
\notag\\
  &+ \|u_{\e,M}\|_{L^{\infty}(\Omega_T)} +  \|\pa_{x_i} u_{\e,M}\|_{L^4(\Omega_T)}  \le C(T),
  \end{align}
and if $N=1$ and $u_{\init} \in W^{1,p}(\Omega)$ for all $p\in [1,\infty)$,
\begin{equation}\label{nest1}
 \| \pa_x u_{\e,M}\|_{L^{p}(\Omega_T)}  \le C(T).
\end{equation}

  In the estimates above, we denote as $C(T)>0$ constants which may depend on $T,\Omega$, on the parameters
  appearing in Assumption A
  (that are $d_v$, $C_f$, $C_f'$, $C_g$, $C_g'$, $a_0$, $a_1$, $a_2$, $a_3$), and on the bounds of the
  initial data, but not on $\e, M$.
\end{lem}
\begin{remark}
	%
%
We observe that the hypothesis on the initial datum $u_{\init}\in L^p(\Omega)$, for any $p\in [1,+\infty)$, is sufficient to prove the regularity of $u_{\e,M}$
the second inequality stated in \eqref{unif eps Linf ve and L4 ue}, and  estimate \eqref{unif eps max regularity ae}. Indeed, the assumption $u_{\init}\in L^{\infty}(\Omega)$ is only used to obtain the $L^{\infty}(\Omega_T)$ boundedness of $u_{\e,M}$ if $N\le 3$ (see \eqref{L2Linf bound dt v lap v}). It is also worth noticing that the hypothesis $v_{\init}\in H^3(\Omega)$ is only used to get estimate \eqref{nez}.
\end{remark}

\begin{proof}
  The first estimate in \eqref{unif eps Linf ve and L4 ue}
  follows from \textit{Proposition~\ref{prop 1 eq in b}} (or
  estimate \eqref{ineq Schauder 1}, under the assumption
  $v_{\init} \in L^{\infty}(\Omega)$).

 We then show the $L^2(\Omega_T)$ boundedness of $\nabla v_{\e,M}$, uniformly in $\e,M$, by multiplying by $v_{\e,M}$ the second equation
  of \eqref{system eps (u,v)} and integrating on $\Omega_t$
  (for any $t \in [0,T]$). Thus, we get
  \begin{equation*}
    \f12\int_{\Omega} v_{\e,M}^2(t)dx+d_v\int_{\Omega_t}|\nabla v_{\e,M}|^2dxds
    \le\f12\int_{\Omega}v^2_{\e,M}(0)dx	+C_g
    \int_{\Omega_t}v_{\e,M}^2dxds\le C(T),
  \end{equation*}
  thanks to the assumption on the initial data and estimate \eqref{unif eps Linf ve and L4 ue}. Hence we end up with
  \begin{equation}\label{unif eps estimate L2 grad ve}
    \|\nabla v_{\e,M}\|_{L^2(\Omega_T)}\le C(T).
  \end{equation}

  We now show the second inequality of \eqref{unif eps Linf ve and L4 ue} and inequality \eqref{unif eps max reg ve}, by proving the following estimate for any $q\in \N-\{0\}$ and $i,j=1,\dots, N,$
    \begin{align}\label{ineq bootstrap}
        \|u_{\e,M}\|_{L^{\infty}([0,T]; L^{2^q}(\Omega))}
        +\|\pa_t v_{\e,M}\|_{L^{2^q}(\Omega_T)}&+\|\pa^2_{x_ix_j}v_{\e,M}\|_{L^{2^q}(\Omega_T)}\notag\\
        &+\|\pa_{x_i} v_{\e,M}\|_{L^{2^{q+1}}(\Omega_T)}\le C(T),
    \end{align}
using an induction on $q$.
More precisely, we will get the boundedness of the first term of the l.h.s. of \eqref{ineq bootstrap} by multiplying the first equation of \eqref{system eps (u,v)} by $u_{\e,M}^{2^q-1}$. Then, the boundedness of the remaining terms in \eqref{ineq bootstrap} will be obtained by maximal regularity and interpolation inequalities.

We prove \eqref{ineq bootstrap} with $q=1$.  Denoting (for any fixed
$\e,M>0$)
    $$A_{\e,M}\coloneqq A(u_{\e,M}, v_{\e,M}), \quad f_{M}\coloneqq f_M(u_{\e,M}, v_{\e,M}), \quad g_{\e,M}\coloneqq g_{\e,M}(u_{\e,M}, v_{\e,M}),$$
    we have by Young's inequality
  \begin{align*}
    \f12 \f d{dt}&\int_{\Omega} u_{\e,M}^2dx   \\
    &=-\int_{\Omega}\pa_1 A_{\e,M}|\nabla u_{\e,M}|^2dx
      -\int_{\Omega}\pa_2 A_{\e,M}\nabla u_{\e,M}\cdot\nabla v_{\e,M} dx
      +\int_{\Omega} u_{\e,M}^2f_Mdx\\
    &\le -\f 12\int_{\Omega} \pa_1 A_{\e,M} |\nabla u_{\e,M}|^2dx
      +\f12\int_{\Omega}\f{|\pa_2A_{\e,M}|^2}{\pa_1 A_{\e,M}}|\nabla v_{\e,M}|^2dx
      +C_f\int_{\Omega}u_{\e,M}^2\,dx\\
    &\le \f{a_2^2}{2a_0}\int_{\Omega}|\nabla v_{\e,M}|^2dx+C_f\int_{\Omega}u^2_{\e,M} \,dx,
  \end{align*}
  by \eqref{hp R1}, \eqref{hp D2}. Then Gronwall's lemma and estimate \eqref{unif eps estimate
    L2 grad ve} yield
  \begin{equation}\label{L2 bound ue unif eps M}
    \|u_{\e,M}\|^2_{L^{\infty}([0,T]; L^2(\Omega))}\le  e^{2 C_f T} \,\|u_{\e,M}(0)\|_{L^{2}(\Omega)}^{2}+ \frac{a_2^2}{a_0}
    e^{2 C_f T}  \, \|\nabla v_{\e,M}\|_{L^2(\Omega_T)}^2 \le C(T) .
  \end{equation}
  Moreover, the first inequality in \eqref{unif eps Linf ve and L4 ue}
  and the obtained estimate \eqref{L2 bound ue unif eps M} imply the
  boundedness of $g_{\e,M}$ in $L^{2}(\Omega_T)$ uniformly in $\e, M$,
  so that the reaction term of the equation of $v_{\e,M}$ in
  \eqref{system eps (u,v)} is bounded in $L^{2}(\Omega_T)$.
  Remembering that by assumption $\nabla v_{\init} \in L^2(\Omega)$,
  maximal regularity and an interpolation inequality ensure that for
  all $i,j=1,\dots N$,
  \begin{equation}\label{L2 unif eps max reg ve}
    \|\pa_t v_{\e,M}\|_{L^2(\Omega_T)}+ \|\pa^2_{x_ix_j}v_{\e,M}\|_{L^2(\Omega_T)}+\|\pa_{x_i}v_{\e,M}\|_{L^4(\Omega_T)}\le  C(T),
  \end{equation}
  thus \eqref{ineq bootstrap} is proved for $q=1$.

We now prove \eqref{ineq bootstrap} for $q \in \N - \{0,1\}$, assuming that it holds for $q-1$. In other words, we assume that for all $i,j=1,\dots, N$
\begin{align}\label{ineq bootstrap q-1}
\|u_{\e,M}\|_{L^{\infty}([0,T]; L^{2^{q-1}}(\Omega))}
        +\|\pa_t v_{\e,M}\|_{L^{2^{q-1}}(\Omega_T)}&+\|\pa^2_{x_ix_j}v_{\e,M}\|_{L^{2^{q-1}}(\Omega_T)}\notag\\
        &+\|\pa_{x_i} v_{\e,M}\|_{L^{2^{q}}(\Omega_T)}\le C(T).
\end{align}
By multiplying the first equation of \eqref{system eps (u,v)}
  by $u^{2^q-1}_{\e,M}$ and integrating on $\Omega$, we get
  \begin{align}\label{ineq u step q}
    &\f 1{2^q} \f{d}{dt} \int_{\Omega} u_{\e,M}^{2^q} dx=-\int_{\Omega}\nabla \big(A_{\e,M}\big)\nabla (u^{2^q-1}_{\e,M})dx+\int_{\Omega} u_{\e,M}^{2^q}f_M
    dx\notag\\
    & \le -(2^q-1)\int_{\Omega} u_{\e,M}^{2^q-2} \, \Big(\pa_1 A_{\e,M}\nabla u_{\e,M}+\pa_2 A_{\e,M}\nabla v_{\e,M}\Big)\cdot\nabla u_{\e,M} dx  +   C_f\int_{\Omega}u_{\e,M}^{2^q}dx \notag\\
	&\le-(2^q-1)\int_{\Omega} \pa_1 A_{\e,M}(u^{2^q-2}_{\e,M})|\nabla u_{\e,M}|^2dx+\f {(2^q-1)}2\int_{\Omega} \pa_1 A_{\e,M}(u^{2^q-2}_{\e,M})|\nabla u_{\e,M}|^2dx\notag\\
   &\quad+\f {(2^q-1)}2\int_{\Omega}\f{|\pa_2 A_{\e,M}|^2}{\pa_1 A_{\e,M}}u_{\e,M}^{2^q-2}|\nabla v_{\e,M}|^2dx +   C_f\int_{\Omega}u_{\e,M}^{2^q}dx\notag\\
   &\le C_q \int_{\Omega} u_{\e,M}^{2^q}dx+ \f 1 {2^{q-1}}\int_{\Omega}|\nabla v_{\e,M}|^{2^q}dx,
 \end{align}
 where we used assumption \eqref{hp D2} and H\"older's inequality with coefficients $\left(\f {2^{q}}{2^{q}-2}, \f{2^{q}}2\right)$.
 Then, Gronwall's lemma, the assumption on $u_{\init} $ 
 and estimate \eqref{ineq bootstrap q-1} yield
 \begin{equation}\label{ineq ue LinfLq}
   \|u_{\e,M}\|_{L^{\infty}([0,T]; L^{2^q}(\Omega))}\le C(T),
 \end{equation}
 giving the boundedness of the first term in \eqref{ineq
   bootstrap}. Moreover, the first inequality in \eqref{unif eps Linf
   ve and L4 ue} and the obtained estimate \eqref{ineq ue LinfLq}
 imply the boundedness of $g_{\e,M}$ in
 $L^{2^q}(\Omega_T)$ uniformly in $\e,
 M$, so that the right hand side of \eqref{system eps (u,v)} is
 bounded in
 $L^{2^q}(\Omega_T)$.  Thus, maximal regularity (which holds thanks to
 the assumption on
 $v_{\init}$) and an interpolation inequality give \eqref{ineq
   bootstrap}.  This concludes the proof of the second inequality in
 \eqref{unif eps Linf ve and L4 ue} and inequality \eqref{unif eps max
   reg ve}.

 We now prove estimate \eqref{unif eps max regularity ae}. First, multiplying by $-\Delta v_{\e,M}$ the equation satisfied by $v_{\e,M}$, we get
 \begin{align*}
\f 12 \f{d}{dt} \int_{\Omega} |\nabla v_{\e, M}|^2dx &+ d_v  \int_{\Omega} |\Delta v_{\e, M}|^2dx
=  \int_{\Omega} \Delta v_{\e, M}\, v_{\e, M}\, g_{\e, M}dx\\
&\le \frac{d_v}2  \int_{\Omega} |\Delta v_{\e, M}|^2dx + \frac{\|v_{\e,M}\|_{L^{\infty}(\Omega)}^2}{2 d_v}
 \int_{\Omega} g_{\e, M}^2 dx,
 \end{align*}
so that integrating w.r.t time and using the uniform $L^2(\Omega_T)$ boundedness of $g_{\e,M}$, and the assumption $v_{\init} \in H^1(\Omega)$, we end up with
 \begin{equation}\label{LinfL2 unif eps M grad v}
      \|\nabla v_{\e,M}\|_{L^{\infty}([0,T]; L^2(\Omega))}\le C(T).
  \end{equation}
Alternatively, one could use the properties of the heat equation to directly obtain that $\|\nabla v_{\e,M}\|_{L^{\infty}(\Omega_T)} \le C_T$, but we will not use this extra information in the sequel.
system \eqref{system regular eps (a, v)}.
 Inequality \eqref{property U/a} and the second estimate in
  \eqref{unif eps Linf ve and L4 ue} (with $p=4$) imply
  \begin{equation}\label{ct1}
    \|a_{\e,M}\|_{L^{\infty}([0,T]; L^4(\Omega))} \le C(T).
  \end{equation}
  Moreover, thanks to estimates \eqref{unif eps Linf ve and L4 ue}, \eqref{unif eps max reg ve} (with $p=4$) we find from Definition (\ref{def s_M}):
\begin{align} \label{ct2}
  \|s_M\|_{L^4(\Omega_T)}
  \le C(T) .
\end{align}
Multiplying the first equation of \eqref{system regular eps (a, v)} by $- \Delta a_{\e, M}$ and integrating on $\Omega_T$, we see that under the
assumption $a_{\init} \in H^1(\Omega)$ (equivalent to
$u_{\init} \in H^1(\Omega)$),
\begin{align}\label{L2 unif esp M dt ae, Delta ae}
  \|\pa_t a_{\e,M}&\|_{L^2(\Omega_T)}^2+\|\nabla
    a_{\e,M}\|_{L^{\infty}([0,T];L^2(\Omega))}^2+\|\Delta
    a_{\e,M}\|_{L^2(\Omega_T)}^2\notag\\
    &\le C(T)\| a_{\e,M}\, s_M\|_{L^2(\Omega_T)}^2
    \le   C(T) \| a_{\e,M} \|_{L^4(\Omega_T)}^2  \,  \|s_M\|_{L^4(\Omega_T)}^2  \le C(T),
\end{align}
thanks to estimates \eqref{ct1} and \eqref{ct2}.  Using identity
\eqref{c of v}, we see that
\begin{equation*}
  \pa_1 A_{\e,M}\pa_t u_{\e,M}=\pa_t a_{\e,M}-\pa_2 A_{\e,M}\pa_t v_{\e,M},
\end{equation*}
so that using \eqref{hp D2}, estimates \eqref{unif eps max reg ve} with $p=2$ and \eqref{L2 unif esp
  M dt ae, Delta ae}, we end up
with
\begin{equation}\label{L2 unif esp M dt ue}
  \|\pa_t u_{\e,M}\|_{L^2(\Omega_T)}^2\le\f 2{a^2_0}\|\pa_t a_{\e,M}\|_{L^2(\Omega_T )}^2+2\Big(\f{a_2}{a_0}\Big)^2\|\pa_t v_{\e,M}\|_{L^2(\Omega_T)}^2\le C(T).
\end{equation}
Similarly,
\begin{equation}\label{ weak nabla A}
  \nabla (A(u_{\e,M},v_{\e,M}))=\pa_1 A_{\e,M}\nabla u_{\e,M}+\pa_2 A_{\e,M}\nabla v_{\e,M},
\end{equation}
so that using \eqref{hp D2} again, \eqref{LinfL2 unif eps M grad v} and \eqref{L2 unif esp M dt ae, Delta ae}, we get
\begin{align}\label{LinfL2 unif eps M grad ue}
  \|\nabla u_{\e,M}\|^2_{L^{\infty}([0,T];L^2(\Omega))}
  &\le C(T),
\end{align}
Finally, estimates
\eqref{L2 unif esp M dt ae, Delta ae}--\eqref{LinfL2 unif
  eps M grad ue} are collected in \eqref{unif eps max regularity
  ae}.
\medskip

In order to prove estimate \eqref{nez}, we take the time
derivative in the second equation of \eqref{system eps (u,v)}. Then, using \eqref{unif eps Linf ve and L4
  ue}--\eqref{unif eps max
  regularity ae} we have
\begin{align}\label{ddt}
    \| \pa_t(\pa_t v_{\e,M})&-d_v\Delta(\pa_t v_{\e,M}) \|_{L^2(\Omega_T)} \notag\\
    &=
      \| (\pa_t v_{\e,M}) g_{\e,M}+v_{\e,M}\pa_1 g_{\e,M}\pa_t u_{\e,M}
      + v_{\e,M}\pa_2 g_{\e,M}\pa_t v_{\e,M} \|_{L^2(\Omega_T)} \notag\\
    &\le \| \pa_t v_{\e,M} \|_{L^4(\Omega_T)} \, \|g_{\e,M}
      \|_{L^4(\Omega_T)} \notag\\
    &\quad +
      \|  v_{\e,M} \|_{L^{\infty}(\Omega_T)}
      \|\pa_1 g_{\e,M} \|_{L^{\infty}(\Omega_T)}  \| \pa_t u_{\e,M}
      \|_{L^2(\Omega_T)} \notag\\
    &\quad +
      \|  v_{\e,M} \|_{L^{\infty}(\Omega_T)}
      \|\pa_2 g_{\e,M} \|_{L^{\infty}(\Omega_T)} \| \pa_t v_{\e,M}
      \|_{L^2(\Omega_T)}\notag\\
    &\le C(T).
\end{align}
Therefore, we apply the maximal regularity, using the assumption $v_{\init}\in H^3(\Omega)$ (and
$u_{\init}, v_{\init} \in L^\infty(\Omega)$) and an interpolation inequality, using estimate \eqref{unif eps max reg ve}, to get for all $i,j=1,..,N$ and for $\eta>0$
 \begin{equation} \label{new246}
   \|\pa^2_{tt} v_{\e,M}\|_{L^2(\Omega_T)}+\|\pa^3_{t x_i x_j}\,v_{\e,M}\|_{L^2(\Omega_T)}+ \|\pa^2_{t x_i} v_{\e,M}\|_{L^{4-\eta}(\Omega_T)}\le C(T),
\end{equation}
that is estimate \eqref{nez}.
\medskip

 We now consider the case when $N\le 3$, so that we can use
 the continuous injection
$ H^2(\Omega)\hookrightarrow L^{\infty}(\Omega)$. We see
that thanks to estimates (\ref{L2 unif esp M dt ae, Delta ae}) and (\ref{new246}),
\begin{equation}\label{L2Linf eps unif dt ve}
 \|A(u_{\e,M}, v_{\e,M})\|_{ L^2([0,T];\,L^{\infty}(\Omega))} +
  \|\pa_t v_{\e,M}\|_{ L^2([0,T];\,L^{\infty}(\Omega))}\le C(T),
\end{equation}
which implies, thanks to Definition \ref{def s_M},
\begin{equation*}
  \int_{0}^{T} \sup_{x\in\Omega} |s_M(t,x)|\,dt\le C(T).
\end{equation*}
Thanks to estimate \eqref{Linf bound b} in \textit{Proposition \ref{prop 1
  eq in b}} applied to the first equation of system (\ref{system regular eps (a, v)}), and the assumption
$a_{\init} \in (L^{\infty} \cap H^1) (\Omega)$ (equivalent to
$u_{\init} \in (L^{\infty} \cap H^1) (\Omega)$), we get
\begin{equation}\label{Linf unif eps ae}
  \|a_{\e,M}\|_{L^{\infty}(\Omega_T)}\le C(T),
\end{equation}
and finally
\begin{equation}\label{Linf unif eps ae u}
  \|u_{\e,M}\|_{L^{\infty}(\Omega_T)}\le C(T).
\end{equation}
Using Gagliardo-Nirenberg's inequality (with the constant in the inequality denoted by $C_{GN}$), we see that for all
$i=1,\dots, N,$ (when $N\le 3$)
\begin{equation*}
  \|\pa_{x_i} a_{\e,M}\|^4_{L^4(\Omega)}
  \le C_{GN} \sup_{k,l=1,\dots ,N} \|\pa^2_{x_kx_l} a_{\e,M}\|_{L^2(\Omega)}^{2}\|a_{\e,M}\|_{L^{\infty}(\Omega)}^{2}+C_{GN} \|a_{\e,M}\|_{L^{\infty}(\Omega)}^4.
\end{equation*}
Then, by integrating in time over $(0,T)$, and using estimate
\eqref{unif eps max regularity ae},  we obtain
\begin{equation}
  \|\pa_{x_i} a_{\e,M}\|_{L^4(\Omega_T)}
  \le C(T) .
\end{equation}

Using now \eqref{hp D2}, \eqref{unif eps max reg ve} with $p=4$, \eqref{ weak
  nabla A} and the above inequality, we end up
for all $i=1,\dots, N$ (and $N\le 3$), with
\begin{align*}
  \|\pa_{x_i} u_{\e,M}\|_{L^4(\Omega_T)}
  &\le C(T)\big(\|\pa_{x_i}  a_{\e,M}\|_{L^{4}(\Omega_T)}+\|\pa_{x_i}  v_{\e,M}\|_{L^{4}(\Omega_T)}\big)\le C(T).
\end{align*}

Finally, we observe that (still when $N \le 3$)
\begin{equation*}
  \| g_{\e,M}(u_{\e,M}, v_{\e,M}) \|_{L^{\infty}(\Omega_T)}
  \le C_g  \| 1+  u_{\e,M} + v_{\e,M} \|_{L^{\infty}(\Omega_T)} \le C(T),
\end{equation*}
thanks to \eqref{Linf unif eps ae u}.  Then, using \eqref{L2Linf eps unif dt ve}, we see that for all
$i,j=1,\dots ,N$,
\begin{equation}\label{est248}
 \| \pa^2_{x_ix_j}  v_{\e,M} \|_{L^2([0,T] ; L^{\infty}(\Omega))} \le C(T).
\end{equation}
Collecting estimates \eqref{L2Linf eps unif dt
  ve}--\eqref{est248}, we obtain \eqref{L2Linf bound dt v lap
  v}.

We finally consider the case when $N=1$. Denoting for simplicity $\mu_{\e,M}=\mu(a_{\e,M}, v_{\e,M})$ and $s_{\e,M}=s_M(a_{\e,M}, v_{\e,M}, \pa_t v_{\e,M} )$, we take $p\ge 1$ and compute (remembering that $a_{\e,M}$ satisfies Neumann's boundary condition)
\begin{align*}
\frac{d}{dt} \int_{\Omega} |\pa_x a_{\e,M}|^{2p} dx  &=  2p  \int_{\Omega} |\pa_x a_{\e,M}|^{2p - 1} \pa_x \big(  \mu_{\e,M} \pa^2_{xx} a_{\e,M} +  a_{\e,M} s_{\e,M}\big)dx \\
& =- 2p \int_{\Omega}  \pa_x \left( |\pa_x a_{\e,M}|^{2p - 1} \right)
   \big( \mu_{\e,M} \pa^2_{xx} a_{\e,M} + a_{\e,M}\, s_{\e,M} \big) dx \\
 &=- 2p(2p-1) \int_{\Omega}   |\pa_x a_{\e,M}|^{2p - 2} \pa^2_{xx} a_{\e,M}
 \big( \mu_{\e,M} \pa^2_{xx} a_{\e,M} + a_{\e,M} s_{\e,M} \big) dx .
  \end{align*}
As a consequence, it holds by \eqref{hp mu}
\begin{align*}
	\frac{d}{dt} \int_{\Omega} |\pa_x a_{\e,M}|^{2p}  dx  &+ 2p(2p-1)a_0 \int_{\Omega}   |\pa_x a_{\e,M}|^{2p - 2} |\pa^2_{xx} a_{\e,M}|^2 dx \\
&\le 2p(2p-1) \int_{\Omega} |\pa_x a_{\e,M}|^{2p - 2} |\pa_{xx} a_{\e,M}|  a_{\e,M} |s_{\e,M} |  dx .
\end{align*}
Thus, by Young's inequality, we end up with
\begin{equation}\label{estimate dt dx a}
	\frac{d}{dt} \int_{\Omega} |\pa_x a_{\e,M}|^{2p}  dx \le \f{p(2p-1)}{a_0} \int_{\Omega} |\pa_x a_{\e,M}|^{2p - 2} \, |a_{\e,M}|^2\, |s_{\e,M} |^2  dx .
\end{equation}
Recalling the uniform bound (see estimates \eqref{hp mu}, \eqref{property U/a}), definition (\ref{def s_M}) and Assumption A)
$$  |s_M(a_{\e,M}, v_{\e,M}, \pa_t v_{\e,M} )| \le \f 1{a_0} \left(a_1\,C_f\, (1+ |U(a_{\e,M}, v_{\e,M})| + v_{\e,M}) + a_3 |\pa_t v_{\e,M}| \right) , $$
we use the estimates  \eqref{unif eps Linf ve and L4 ue} and \eqref{Linf unif eps ae u} to get
$$  |s_M(a_{\e,M}, v_{\e,M}, \pa_t v_{\e,M} )|^2 \le C(T)\,( 1+   |\pa_t v_{\e,M}|^2). $$
Therefore using \eqref{Linf unif eps ae}, the right-hand side of \eqref{estimate dt dx a} is estimated as
\[
	\frac{d}{dt} \int_{\Omega} |\pa_x a_{\e,M}|^{2p} \, dx \le p\,(2p-1) \, C(T)\|a_{\e,M}\|_{L^\infty(\Omega)}^2\int_{\Omega} |\pa_x a_{\e,M}|^{2p - 2} \, ( 1+   |\pa_t v_{\e,M}|^2) \, dx.
\]
Finally, using Hölder's inequality with coefficients $\left(\f{2p}{2p-2}, \f{2p}2\right)$ for any $p\ge1$, we get
\[
\frac{d}{dt}\|\pa_x a_{\e,M}\|^{2p}_{L^{2p}(\Omega)} \, dx \le C_{p}(T)\|a_{\e,M}\|_{L^\infty(\Omega)}^2\left(1+\|\pa_tv_{\e,M}\|_{L^{2p}(\Omega)}^{2}\right)\|a_{\e,M}\|_{L^{2p}(\Omega)}^{2p-2}.
\]
 Thus, we obtain estimate (\ref{nest1}) by Duhamel's formula, recalling that we assumed  that $u_{in} \in W^{1,p}(\Omega)$ for all $p\in [1,\infty)$, and using estimates \eqref{unif eps max reg ve}.
\end{proof}

We now conclude the proof of \textit{Theorem \ref{thm existence (u,v) + regularity}}.

By the $\e,M-$uniform estimates shown in
\textit{Lemma~\ref{lll}}, we can extract subsequences from $u_{\e,M}$, $v_{\e,M}$ (still denoted by $u_{\e,M}$, $v_{\e,M}$) such that for
some $u\in L^{\infty}([0,T]; L^p(\Omega))$, for all $p\in [1,+\infty)$, and
$v\in L^{\infty}(\Omega)$,
\begin{equation}\label{limnit a.e. ue ve}
  u_{\e,M}\rightarrow u,\qquad v_{\e,M}\rightarrow v,\qquad\text{ a.e. in $\Omega_T$ \quad as }\e\rightarrow 0,\, M\rightarrow +\infty,
\end{equation}
and for all $ p\in [1, \infty)$
\begin{align}\label{limit dt ue dt and grad ve}
    \pa_t u_{\e,M}\rightharpoonup\pa_t u,\qquad\qquad\quad
    \qquad&\text{weakly in $L^2(\Omega_T),$}\notag\\
    \pa_t v_{\e,M}\rightharpoonup\pa_t v,\,\,
    \Delta v_{\e,M}\rightharpoonup\Delta v,\,\,
    \nabla v_{\e,M}\rightharpoonup\nabla v,\quad&\text{weakly in } L^p(\Omega_T).
\end{align}
Moreover, by Assumption~A and estimates \eqref{unif eps
  Linf ve and L4 ue}, \eqref{unif eps max regularity ae}, the
convergence \eqref{limnit a.e. ue ve} ensures that
\begin{equation}\label{conv Delta A eps wL2}
  \Delta A\big(u_{\e,M},v_{\e,M}\big)\rightharpoonup\Delta A(u,v),\qquad \text{weakly in $L^2(\Omega_T)$}.
\end{equation}

Now, we take the $D'(\Omega_T)$ limit as
$\e\rightarrow 0,\,M\rightarrow +\infty$, in \eqref{system eps (u,v)},
\eqref{def initial data ue ve}.  and observe that
$u_{\e,M}f_M(u_{\e,M}, v_{\e,M})$ and
$v_{\e,M}g_{\e,M}(u_{\e,M}, v_{\e,M})$ converge towards $u\,f(u,v)$
and $v\,g(u,v)$ strongly in $L^1(\Omega_T)$, thanks to assumption
\eqref{hp R1} and the estimates \eqref{unif eps Linf ve and L4 ue}.
\medskip

Then, using the convergences obtained above, all the terms in the first
two equations of \eqref{system eps (u,v)} converge in
$D'(\Omega_T).$ We conclude by taking the limit in the boundary
conditions of \eqref{system eps (u,v)}, using the continuity of
the trace operator and the weak convergence of
$\Delta v_{\e,M},\, \Delta A(u_{\e,M}, v_{\e,M})$ in
\eqref{limit dt ue dt and grad ve}, \eqref{conv Delta A
  eps wL2}.  Finally, using the lower semi-continuity property of the
$L^p(\Omega_T)$ norm for $p\in (1,+\infty]$, we conclude that $u,v$
satisfy the estimates stated in lines i) and ii) of  \textit{Theorem~\ref{thm
    existence (u,v) + regularity}}.

\section{Stability and uniqueness}\label{sec:uniqueness}

In this section, we present the proof of the two stability results
(\textit{Theorem \ref{thm uniqueness}} and \textit{Theorem
  \ref{nuq}}).

For a better readability, we first introduce for \(i=1,2\) the notations
\begin{equation*}
  \begin{split} A_i
    &\coloneqq A(u_i,v_i),\qquad f_i\coloneqq
      f(u_i,v_i),\qquad g_i\coloneqq
      g(u_i,v_i),\\
    &\partial_1
      A_i\coloneqq \partial_1 A(u_i,v_i),\qquad \partial_2 A_i\coloneqq
      \partial_2 A(u_i,v_i)
  \end{split}
\end{equation*}
and
\begin{equation*}
  \mathcal{N}_H := \sup_{i,j=1,2} \|\pa^2_{ij} A \|_{\infty}.
\end{equation*} Hereafter, all constants $C$ are strictly positive
and may change from line to line.  Moreover, if foreseen, the
dependency of parameter to be chosen thereafter is denoted as index of
the constant.

\begin{proof}[Proof of Theorem~\ref{thm uniqueness}]
 We compute the equations satisfied by $u_1-u_2$ and $v_1-v_2$
  and we multiply by $u_1-u_2$ and $\lambda(v_1-v_2),$ respectively,
  where the parameter $\lambda>0$ will be chosen later. Then, we
  integrate over $\Omega$ and we add the obtained formulations to get
  \begin{align}\label{def Irea Idiff uniqueness}
    \f12
    &\f d{dt}\Big(\int_{\Omega}|u_1-u_2|^2dx+\lambda\int_{\Omega}
    |v_1-v_2|^2dx\Big)\notag\\
    &=-\int_{\Omega}\big(\partial_1 A_1\nabla u_1+\partial_2 A_1\nabla v_1\big)\cdot\nabla (u_1-u_2)dx\notag\\
    &\quad+\int_{\Omega}
      \Big(\partial_1A_2\nabla u_2+\partial_2 A_2\nabla v_2\Big)\cdot \nabla(u_1-u_2)dx-d_v\lambda\int_{\Omega}|\nabla(v_1-v_2)|^2dx\notag\\
    &\quad+\int_{\Omega}\big(u_1f_1-u_2f_2\big)(u_1-u_2) dx+\lambda\int_{\Omega}\big(v_1g_1-v_2g_2\big)(v_1-v_2) dx\notag\\
    &\eqqcolon I_{diff}+I_{rea}.
  \end{align}
  The reaction part is then estimated as below
  \begin{align}\label{reaction estimate prov}
    I_{rea}
    &=\int_{\Omega}f_1|u_1-u_2|^2dx+\int_{\Omega}u_2(f_1-f_2)(u_1-u_2)dx\notag\\
    &\quad+\lambda\int_{\Omega}g_1|v_1-v_2|^2dx+\lambda\int_{\Omega}v_2(g_1-g_2)(v_1-v_2)dx\notag\\
    &\le \max(C_{f}, C_{g})  \, \Big(\int_{\Omega}|u_1-u_2|^2dx+\lambda\int_{\Omega}|v_1-v_2|^2dx\Big)\notag\\
   &\quad + C_f' \, \|u_2\|_{L^{\infty}} \int_{\Omega}
     \bigg[|u_1-u_2|^2 + |u_1-u_2|\, |v_1-v_2| \bigg]\,  dx \notag\\
   &\quad
+ \lambda  \,  C_g' \, \|v_2\|_{L^{\infty}}  \int_{\Omega}
     \bigg[|v_1-v_2|^2 + |u_1-u_2|\, |v_1-v_2| \bigg]\,  dx ,
  \end{align}
  thanks to \eqref{hp R1}.  Using Young's inequality, we get from
  \eqref{reaction estimate prov}
  \begin{align}\label{final react estimate}
    I_{rea}
    &\le	C_{\lambda}(1 + \|u_2\|_{L^{\infty}}+  \|v_2\|_{L^{\infty}})
    \left(\|u_1-u_2\|_{L^2(\Omega)}^2+\lambda\|v_1-v_2\|_{L^2(\Omega)}^2\right).
  \end{align}

  Concerning the diffusion part, it holds by Young's inequality and \eqref{hp D2},
  \begin{align}\label{est Idiff uniq}
	I_{diff}&=-\int_{\Omega}\partial_1 A_1|\nabla (u_1-u_2)|^2dx-d_v\lambda\int_{\Omega}|\nabla (v_1-v_2)|^2dx\notag\\
	&\quad-\int_{\Omega}\partial_1 (A_1-A_2)\nabla u_2\cdot \nabla(u_1-u_2)dx
	-\int_{\Omega}\partial_2 A_1\nabla(v_1-v_2)\cdot\nabla(u_1-u_2)dx\notag\\
	&\quad-\int_{\Omega}\partial_2\big(A_1-A_2\big)\nabla v_2\cdot\nabla (u_1-u_2)dx\notag\\
	&\le -\f{a_0}4\int_{\Omega}|\nabla (u_1-u_2)|^2dx-\big(d_v\lambda-\f{a_2^2}{a_0}\big)\int_{\Omega}|\nabla (v_1-v_2)|^2dx\notag\\
	&\quad + \int_{\Omega}\f{|\nabla u_2|^2}{\partial_1 A_1}|\partial_1 (A_1-A_2)|^2dx + \int_{\Omega}\f{|\nabla v_2|^2}{\partial_1 A_1}|\partial_2 (A_1-A_2)|^2dx.
  \end{align}
  We now focus on the last two integrals in \eqref{est Idiff uniq}. The second one is estimated
  as follows
  \begin{align}\label{est grad v N=1,2}
	\int_{\Omega}&\f{|\nabla v_2|^2}{\partial_1A_1}|\partial_2 (A_1-A_2)|^2dx\notag\\
	&\le\f2{a_0}\, \mathcal{N}_H\, \|\nabla v_2\|_{L^{\infty}(\Omega)}^2\int_{\Omega}\Big(|u_1-u_2|^2+ |v_1-v_2|^2\Big)dx\notag\\
	&\le C_\lambda\|\nabla v_2\|_{L^{\infty}(\Omega)}^2
	\int_{\Omega}\Big(|u_1-u_2|^2+\lambda|v_1-v_2|^2\Big)dx.
  \end{align}

  In order to estimate the first integral, we use \eqref{hp D2} and
  get
  \begin{align}\label{first diff term N=2}
	\int_{\Omega}&\f{|\nabla u_2|^2}{\partial_1 A_1}|\partial_1 (A_1-A_2)|^2dx \notag\\
	&\le\f2{a_0} \|\nabla u_2\|_{L^4(\Omega)}^2  \,  \|\partial_1 (A_1-A_2)\|_{L^4(\Omega)}^2 \notag\\
	&\le \f2{a_0}\mathcal{N}^2_H \|\nabla u_2\|_{L^4(\Omega)}^2\Big(\|u_1-u_2\|_{L^4(\Omega)}^2+\|v_1-v_2\|_{L^4(\Omega)}^2\Big).
  \end{align}
  Then, Gagliardo-Nirenberg's inequality in dimension 2
  \cite{Nirenberg} allows us to estimate the $L^4$ norm of $(u_1-u_2)$
  (resp. $(v_1-v_2)$) in terms of the $L^2$ norm of $(u_1-u_2)$
  (resp. $(v_1-v_2)$) and $\nabla (u_1-u_2)$ (respectively
  $\nabla (v_1-v_2)$), as follows
  \begin{align*}
    &\|\nabla u_2\|_{L^4(\Omega)}^2\|u_1-u_2\|^2_{L^4(\Omega)}\notag\\
    &\le C_{GN}\|\nabla u_2\|_{L^4(\Omega)}^2\big(\|\nabla (u_1-u_2)\|_{L^2(\Omega)}\|u_1-u_2\|_{L^2(\Omega)}+\|u_1-u_2\|^2_{L^2(\Omega)}\big)\notag\\
    &\le \delta C_{GN}\|\nabla (u_1-u_2)\|_{L^2(\Omega)}^2+\f{C_{GN}}{\delta}\|\nabla u_2\|_{L^4(\Omega)}^4\|u_1-u_2\|_{L^2(\Omega)}^2\notag\\
    &\quad+C_{GN}\|\nabla u_2\|_{L^4(\Omega)}^2\|u_1-u_2\|_{L^2(\Omega)}^2\notag\\
    &\le\delta C_{GN}\|\nabla (u_1-u_2)\|_{L^2(\Omega)}^2
      + C_\delta (1+\|\nabla u_2\|^4_{L^4}) \|u_1-u_2\|_{L^2(\Omega)}^2,
  \end{align*}
  where we denote by $C_{GN}$ the constant involved in the
  Gagliardo-Nirenberg inequality, $\delta>0$ a parameter to be chosen
  later, (and $C_\delta$ a constant depending on $\delta$).
  Similarly, for the second term in \eqref{first diff term N=2} it
  holds
  \begin{align*}
    \|\nabla u_2\|_{L^4(\Omega)}^2\|v_1-v_2\|^2_{L^4(\Omega)}
    \le\delta C_{GN}\|\nabla (v_1-v_2)\|_{L^2(\Omega)}^2
    +C_\delta (1+ \|\nabla u_2\|^4_{L^4})\|v_1-v_2\|_{L^2(\Omega)}^2,
  \end{align*}
  so that \eqref{first diff term N=2} is estimated as
  \begin{align}\label{final est 1st term diff N=2}
	\int_{\Omega}\f{|\nabla u_2|^2}{\partial_1 A_1}|\partial_1 (A_1-A_2)|^2dx&\le \f {2\delta}{a_0}C_{GN}\mathcal{N}^2_H\Big(\|\nabla (u_1-u_2)\|_{L^2(\Omega)}^2+\|\nabla (v_1-v_2)\|_{L^2(\Omega)}^2\Big)\notag\\
	&+C_{\delta,\lambda} (1+\|\nabla u_2\|^4_{L^4}) \big(\|u_1-u_2\|_{L^2(\Omega)}^2+\lambda\|v_1-v_2\|_{L^2(\Omega)}^2\big).
  \end{align}

  Therefore, gathering \eqref{est grad v N=1,2} and \eqref{final est
    1st term diff N=2} into \eqref{est Idiff uniq}, the term $I_{diff}$ is estimated as, renaming the constants,
  \begin{align}\label{final est Idiff N=2}
	I_{diff}&\le -\Big(\f{a_0}4-\f{2\delta}{a_0}C_{GN}\mathcal{N}^2_H\Big)\|\nabla (u_1-u_2)\|_{L^2(\Omega)}^2\notag\\
	&\quad-\Big(d_v\lambda-\f{a_2^2}{a_0}-\f{2\delta}{a_0}C_{GN}\mathcal{N}^2_H\Big)\|\nabla (v_1-v_2)\|_{L^2(\Omega)}^2\notag\\
	&\quad+C_{\delta,\lambda} (1 + \|\nabla u_2\|^4_{L^4} + \|\nabla v_2\|^2_{L^{\infty}})
	\big(\|u_1-u_2\|_{L^2(\Omega)}^2+\lambda\|v_1-v_2\|_{L^2(\Omega)}^2\big).
  \end{align}

  Finally, plugging \eqref{final react estimate},
  \eqref{final est Idiff N=2} into \eqref{def Irea Idiff
    uniqueness}, we end up with
  \begin{equation}\label{final est N=2}
    \begin{split}
      &\f12
      \f d{dt}\Big(\|u_1-u_2\|_{L^2(\Omega)}^2+\lambda\|v_1-v_2\|_{L^2(\Omega)}^2\Big)\\
      &\le -\Big(\f{a_0}4-\f{2\delta}{a_0}C_{GN}\mathcal{N}^2_H\Big)\|\nabla (u_1-u_2)\|_{L^2(\Omega)}^2\\
      &\quad-\Big(d_v\lambda-\f{a_2^2}{a_0}-\f{2\delta}{a_0}C_{GN}\mathcal{N}^2_H\Big)\|\nabla (v_1-v_2)\|_{L^2(\Omega)}^2\\
      &\quad+ C_{\delta, \lambda}\,
        (1+\|\nabla u_2\|^4_{L^4} + \|\nabla v_2\|^2_{L^{\infty}}+\|
        u_2\|_{L^{\infty}}^2 + \| v_2\|_{L^{\infty}}^2)\\
      &\quad\qquad
      \big(\|u_1-u_2\|_{L^2(\Omega)}^2+\lambda\|v_1-v_2\|_{L^2(\Omega)}^2\big).
    \end{split}
  \end{equation}
  Now, we pick $\delta\in (0,\f{a_0^2}{8C_{GN}\mathcal{N}^2_H})$ and
  $\lambda \in (\f{a_2^2}{a_0d_v}+\f{a_0}{4d_v}, \infty)$, and obtain, for some $c>0$,
  \begin{align*}
    \f d{dt}\big(\|u_1-u_2\|_{L^2(\Omega)}^2+\lambda\|v_1
    -v_2\|_{L^2(\Omega)}^2\big)&
    + c \,\bigg(\|\nabla(u_1-u_2)\|_{L^2(\Omega)}^2+\|\nabla(v_1-v_2)\|_{L^2(\Omega)}^2 \bigg)\notag\\
    &\le C\,\big(\|u_1-u_2\|_{L^2(\Omega)}^2+\lambda\|v_1-v_2\|_{L^2(\Omega)}^2\big),
  \end{align*}
  where $C$ depends on
 $\|u_2\|_{L^{\infty}(\Omega_T)}$,
  $\|\nabla u_2\|_{L^{\infty}([0,T]; L^4(\Omega))}$,
  $\|v_2\|_{L^{\infty}(\Omega_T)}$,
  $\|\nabla v_2\|_{L^{\infty}(\Omega_T)}$.

  Finally, we get \eqref{uniqueness-ineq-prop} thanks to Gronwall's lemma.
\end{proof}

We now present the proof of stability stated in \textit{Theorem
  \ref{nuq}} by looking at the evolution in time of the
\((H^1)'(\Omega)\) norm.

\begin{proof}[Proof of Theorem~\ref{nuq}]
  Let consider $\phi=\phi(t,x)$ as the unique solution to the Neumann problem, for all $t\in (0,T)$,
  \begin{equation}\label{elliptic problem phi}
    - \Delta \phi = (u_1-u_2) -  (u_1-u_2)_{\Omega} \,\, {\hbox{ in}}\,\Omega,\qquad \nabla \phi \cdot \sigma = 0 \,\,{\hbox{ on}}\,\pa\Omega, \qquad  (\phi)_{\Omega} =0.
  \end{equation}
  and $\psi=\psi(t,x)$ as the unique solution to the Neumann problem, for all $t\in (0,T)$,
  \begin{equation}\label{elliptic problem psi}
    - \Delta \psi = (v_1 - v_2) -  (v_1-v_2)_{\Omega} \,\, {\hbox{ in}}\,\Omega,\qquad \nabla \psi \cdot \sigma = 0 \,\, {\hbox{ on}}\,\pa\Omega, \qquad  (\psi)_{\Omega} =0.
  \end{equation}
  We compute
  \begin{align}
    \f12\f d{dt}\int_{\Omega} |\nabla \phi|^2dx
    &=\f 12\f d{dt}\int_{\Omega}\left(-\Delta\phi +(u_1-u_2)_{\Omega}\right)dx=\f 12\f d{dt}\int_{\Omega}\phi(u_1-u_2)dx\notag\\
    &=\f 12\int_{\Omega}\f {d}{dt}\phi \left(-\Delta\phi +(u_1-u_2)_{\Omega}\right)dx+\f 12\int_{\Omega} \phi \f d{dt}(u_1-u_2)dx\notag\\
    &=-\f 12\int_{\Omega} \left(\Delta\f d{dt} \phi \right)\phi dx+\f 12\int_{\Omega} \phi \f d{dt}(u_1-u_2)dx=\int_{\Omega} \phi \f d{dt} (u_1-u_2)dx\notag\\
    &= \int_{\Omega} \phi  \Delta (A_1-A_2) dx + \int_{\Omega}\phi (u_1f_1 - u_2 f_2  )dx\eqqcolon I_{diff}+I_{rea}.\label{hm1}
  \end{align}
  Using the first equation in \eqref{elliptic problem phi}, the assumption \eqref{hp D2} and Young's inequality, we estimate
  \begin{align}\label{est Idiff u weak uniq}
    I_{diff}&=\int_{\Omega}\left(A(u_1,v_1)-A(u_2,v_1)\right)\left(-(u_1-u_2)+(u_1-u_2)_{\Omega}\right)dx\notag\\
            &\quad+\int_{\Omega}\left(A(u_2,v_1)-A(u_2,v_2)\right)\left(-(u_1-u_2)+(u_1-u_2)_{\Omega}\right)dx\notag\\
            &\le -a_0\int_{\Omega}(u_1-u_2)^2dx+a_1|\Omega|(u_1-u_2)^2_{\Omega}\notag\\
            &\quad + a_2 \int_{\Omega} |u_1-u_2| |v_1-v_2|dx+\f{a_2}2|\Omega|\left((u_1-u_2)^2_{\Omega}+(v_1-v_2)^2_{\Omega}\right)\notag\\
            &\le -a_0\left(1-\f1{8}\right)\int_{\Omega}(u_1-u_2)^2dx+(a_1+a_2)|\Omega|(u_1-u_2)^2_{\Omega}\notag\\
            &\quad + \f{2a^2_2}{a_0} \int_{\Omega} (v_1-v_2)^2dx+ a_2|\Omega|(v_1-v_2)^2_{\Omega},
  \end{align}
  and by \eqref{hp R1}, we estimate for any $M >0$ to be determined later
  \begin{align}
    I_{rea}&=\int_{\Omega}\phi f_1(u_1-u_2)dx+\int_{\Omega}\phi u_2(f(u_1,v_1)-f(u_1,v_2)+f(u_1,v_2)-f(u_2,v_2))dx\notag \\
           &\le \f{a_0}{8}\int_{\Omega}(u_1-u_2)^2dx + \f{4}{a_0} C^2_f\left(1+\|u_1\|^2_{L^{\infty}(\Omega)}+\|v_1\|^2_{L^{\infty}(\Omega)}\right)\int_{\Omega} \phi^2dx \notag\\
           &\quad+ C_f'\|u_2\|_{L^{\infty}(\Omega)}\int_{\Omega}|\phi|\left(|u_1-u_2|+|v_1-v_2|\right)dx\notag\\
           &\le \f{a_0}{4}\int_{\Omega}(u_1-u_2)^2dx + \f{d_v M}{8}\int_{\Omega}(v_1-v_2)^2dx+ C_1\int_{\Omega}\phi^2dx,\label{est Irea u weak uniq}
  \end{align}
  with $C_1$ depending on $a_0$, $d_v$, $M$, $|\Omega|$, $C_f$, $C_f'$, $\|u_i\|_{L^{\infty}(\Omega)}, \|v_i\|_{L^{\infty}(\Omega)}$, $i=1,2$.
  Then, we put \eqref{est Idiff u weak uniq} and \eqref{est Irea u weak uniq} into \eqref{hm1}
  to get
  \begin{align}\label{evolution phi}
    \frac 12 \frac{d}{dt} \int_{\Omega} |\nabla \phi|^2dx
    &\le -\f {5a_0} 8  \int_{\Omega}(u_1-u_2)^2dx + \left(\f{d_vM}{8}+\f{2a_2^2}{a_0}\right)\int_{\Omega}(v_1-v_2)^2dx\notag\\
    &\quad +C_1\int_{\Omega}\phi^2dx+(a_1+a_2)|\Omega|(u_1-u_2)_{\Omega}^2+a_2|\Omega|(v_1-v_2)_{\Omega}^2.
  \end{align}
  Similarly as in \eqref{hm1}, using the first equation in \eqref{elliptic problem psi}, we have for any $M>0$
  \begin{align*}
    \f M2&\f d{dt}\int_{\Omega}|\nabla \psi|^2dx\\
         &=d_vM\int_{\Omega}(v_1-v_2)\left(-(v_1-v_2)+(v_1-v_2)_{\Omega}\right)dx
           +\int_{\Omega}M\psi(v_1g_1-v_2g_2)dx\\
         &\eqqcolon II_{diff}+II_{rea},
  \end{align*}
  with
  \begin{equation*}
    II_{diff}\le -d_vM\left(1-\f1{8}\right)\int_{\Omega} (v_1-v_2)^2dx+ 2d_vM |\Omega|(v_1-v_2)_{\Omega}^2,
  \end{equation*}
  and
  \begin{equation*}
    II_{rea}\le \f{a_0}{8}\int_{\Omega} (u_1-u_2)^2dx + \f{d_vM}{4}\int_{\Omega}(v_1-v_2)^2dx+C_2\int_{\Omega}\psi^2dx,
  \end{equation*}
  with $C_2$ depending on $a_0$, $d_v$, $M$, $|\Omega|$, $C_g$, $C_g'$,
  $\|u_i\|_{L^{\infty}(\Omega)}, \|v_i\|_{L^{\infty}(\Omega)}$, $i=1,2$.
  Therefore, we end up with
  \begin{align}\label{evolution psi}
    \f M2 \f d{dt}\int_{\Omega}|\nabla \psi|^2dx
    &\le -\f 5 8 d_vM\int_{\Omega}(v_1-v_2)^2dx+ \f{a_0}{8}\int_{\Omega}(u_1-u_2)^2dx\notag\\
    &\quad+ C_2\int_{\Omega}\psi^2dx+2d_v M|\Omega|(v_1-v_2)_{\Omega}^2.
  \end{align}
  By gathering \eqref{evolution phi}, \eqref{evolution psi}, we obtain
  \begin{multline}\label{evolution phi+psi}
    \frac 12 \frac{d}{dt}
    \int_{\Omega}\left( |\nabla \phi|^2 +M |\nabla
      \psi|^2\right)dx\\
    \le -\f{a_0}2\int_{\Omega}(u_1-u_2)^2dx - \left(\f{d_vM}{2}-\f{2a_2^2}{a_0}\right)\int_{\Omega}(v_1-v_2)^2dx
    + (C_1+C_2)\int_{\Omega}(\phi^2+\psi^2)dx\\+(a_1+a_2)|\Omega|(u_1-u_2)^2_{\Omega}+(a_2+d_vM)|\Omega|(v_1-v_2)^2_{\Omega}.
  \end{multline}

  Now, we analyse the evolution in time of $(u_1-u_2)_{\Omega}$ and $(v_1-v_2)_{\Omega}$, respectively.
  \begin{align}\label{def J1 J2}
    \f 12 \f d{dt}(u_1-u_2)_{\Omega}^2
    &=
      \f 1{|\Omega|^2} \int_{\Omega}(u_1-u_2)dx\int_{\Omega}(u_1f_1-u_2f_2)dx\notag\\
    &=\f 1{|\Omega|^2}\int_{\Omega}(u_1-u_2)dx\int_{\Omega}f_1(u_1-u_2)dx\notag\\
    &\quad+\f 1{|\Omega|^2}\int_{\Omega}(u_1-u_2)dx\int_{\Omega}u_2(f_1-f_2)dx\eqqcolon J_1+J_2,
  \end{align}
  thus using \eqref{hp R1}, we compute by Young's inequality and H\"older's inequality
  \begin{align}\label{estimate J1 average u}
    J_1&\le \f {a_0}{8|\Omega|}\left(\int_{\Omega}(u_1-u_2)dx\right)^2+\f{2}{a_0|\Omega|^3}\left(\int_{\Omega}f_1(u_1-u_2)dx\right)^2\notag\\
       &\le \f{a_0}{8}\|u_1-u_2\|_{L^2(\Omega)}^2+C'(u_1-u_2)_{\Omega}^2,
  \end{align}
  and  similarly
  \begin{align}\label{estimate J2 average u}
    J_2&=\f 1{|\Omega|^2}\int_{\Omega}(u_1-u_2)dx\int_{\Omega}u_2\left(f(u_1,v_1)-f(u_1,v_2)+f(u_1,v_2)-f(u_2,v_2)\right)dx\notag\\
       &\le \f{a_0}{8}\|u_1-u_2\|_{L^2(\Omega)}^2+\f{d_v M}{8}\|v_1-v_2\|_{L^2(\Omega)}^2+C''(u_1-u_2)_{\Omega}^2.
  \end{align}
  By putting \eqref{estimate J1 average u}, \eqref{estimate J2 average u} into \eqref{def J1 J2}, we obtain
  \begin{align}\label{evolution average u}
    \f 12 \f d{dt}(u_1-u_2)_{\Omega}^2&\le \f{a_0}{4}\|u_1-u_2\|_{L^2(\Omega)}^2+\f{d_v M}{8}\|v_1-v_2\|_{L^2(\Omega)}^2 +C_3(u_1-u_2)_{\Omega}^2,
  \end{align}
  where the constant $C_3$ depends on $a_0, d_v, \eta_u, \eta_v, M, \Omega$ and on $\|u_i\|_{L^{\infty}(\Omega)}, \|v_i\|_{L^{\infty}(\Omega)}$, $i=1,2$. Similarly, it holds
  \begin{align}\label{evolution average v}
    \f 12 \f d{dt}(v_1-v_2)_{\Omega}^2 &\le \f{a_0}{8}\|u_1-u_2\|_{L^2(\Omega)}^2+\f{d_v M}{4}\|v_1-v_2\|_{L^2(\Omega)}^2+C_4(v_1-v_2)_{\Omega}^2,
  \end{align}
  where the constant $C_4$ depends on $a_0, d_v, \eta_u, \eta_v, M, \Omega$ and on $\|u_i\|_{L^{\infty}(\Omega)}, \|v_i\|_{L^{\infty}(\Omega)}$, $i=1,2$.

  Adding estimates (\ref{evolution phi+psi}), (\ref{evolution average
    u}), (\ref{evolution average v}), using Poincar\'e's inequality
  (remember that
  $ \int_{\Omega} \phi\, dx = \int_{\Omega} \psi\, dx =0$) and
  renaming the constants, we get
  \begin{align*}
    \frac 12 \frac{d}{dt} \int_{\Omega}
    &\left( |\nabla \phi|^2  + M |\nabla \psi|^2+ (u_1-u_2)^2_{\Omega} +(v_1-v_2)_{\Omega}^2\right)dx\\
    &\le -\f{a_0}8  \|u_1-u_2\|_{L^2(\Omega)}^2-\left(\f{d_vM}8-\f{2a_2^2}{a_0}\right)\|v_1-v_2\|_{L^2(\Omega)}^2\\
    &\quad+ C_M\int_{\Omega}\left( |\nabla \phi|^2  + M |\nabla \psi|^2+ (u_1-u_2)^2_{\Omega} +(v_1-v_2)_{\Omega}^2\right)dx.
  \end{align*}
  Thus, taking $M>0$ large enough, this finishes the proof by
  Gronwall's lemma.

  Finally, we remark how the dependency on the first solution can be
  weakened and explain it on the first term of \(I_{rea}\).  With
  \(p\) satisfying \(1/p > 1/2 - 1/N\) and the conjugate index
  \(p^*\), we can estimate the term as
  \begin{align*}
    \int_{\Omega}\phi f_1(u_1-u_2)dx
    &\le \| \phi \|_p \, \| f_1(u_1-u_2) \|_{p^*} \\
    &\le \| \phi \|_p \, \| f_1 \|_r \, \| u_1-u_2 \|_{2},
  \end{align*}
  where \(1/r = 1/p^* - 1/2\).  Then the stability constant only
  depends on \(\|u_1\|_r\) as by the choice of \(1/p > 1/2 - 1/N\) we
  can estimate for \(\alpha \in (0,1)\)
  \begin{equation*}
    \| \phi \|_p
    \lesssim \| \phi \|_2^\alpha \| \nabla \phi \|_2^{1-\alpha}
    + \| \phi \|_2
  \end{equation*}
  so that the Gronwall estimate can be closed.
\end{proof}

\appendix

\section{Some classical results for linear parabolic equation}\label{app-sect:esistence}

We state here the following existence result for a linear parabolic
equation in \textit{non-divergence} form, together with some standard estimates.

\begin{prop}\label{prop 1 eq in b}
  We consider the following linear parabolic problem defined on a smooth bounded domain $\Omega$ of $\R^N$:
  \begin{equation}\label{eq in b}
    \begin{cases}
      \partial_t b-\gamma(t,x)\Delta b= r(t,x)\,b,\qquad&\text{ in } \Omega_T,\\
      \nabla b\cdot\sigma=0,\qquad&\text{ on }(0,T)\times\pa\Omega,\\
      b(0,x)=b_{\init}(x)\ge0,\qquad&\text{ on }\Omega,
    \end{cases}
  \end{equation}
  where
  \begin{enumerate}[label={\roman*)}]
  \item $ \gamma:\Omega_T\rightarrow\R_+,$ $\gamma\,$ lies in
    $\,L^{\infty}([0,T]  ; W^{1,\infty}(\Omega))$
    and there exist two constants
    $\gamma_0,\gamma_1>0$ such that
    \begin{equation}\label{hp gamma}
      0<\gamma_0\le\gamma (t,x)\le\gamma_1,\quad \text{a.e. in }\Omega_T,
    \end{equation}
  \item $r:\Omega_T\rightarrow\R,$ $r$ lies in $L^2(\Omega_T)$ and $L^{1}([0,T]; L^{\infty}(\Omega))$,
  \item $b_{\init}:\Omega\rightarrow \R_+$ satisfies
    \begin{equation}\label{hp b init}
      b_{\init}\in \big(L^{\infty}\cap H^1\big)(\Omega).
    \end{equation}
  \end{enumerate}
  Then, there exists a nonnegative strong (in the sense of
  Theorem~\ref{thm existence (u,v) + regularity}) solution $b$
  to \eqref{prop 1 eq in b} such that
  \begin{enumerate}[label={\roman*)}]
  \item for all $t\in(0,T)$, b satisfies
    \begin{equation}\label{Linf bound b}
      \|b\|_{L^{\infty}(\Omega_T)}\le \|b_{\init}\|_{L^{\infty}(\Omega)}e^{\int_0^T \sup\limits_{x \in \Omega} r(t,x) \, dt},
    \end{equation}
  \item there exists a constant $C>0$ depending on $b_{\init}, \gamma_0, \gamma_1$, such that
    \begin{equation}\label{estimate b}
      \|\partial_t b\|_{L^2(\Omega_T)}^2+\|\nabla b\|^2_{L^{\infty}([0,T];L^2(\Omega))}+\|\Delta b\|_{L^2(\Omega_T)}^2
      \le C\left(1+\|rb\|_{L^2(\Omega_T)}^2\right).
    \end{equation}
  \end{enumerate}
\end{prop}

\medskip

We briefly explain the strategy of proof of \textit{Proposition \ref{prop 1 eq
  in b}}: \medskip

We first
consider a regularized version of \eqref{eq in b}, which admits a
 unique classical solution, thanks to classical results of parabolic PDEs theory
\cite{LSU}.
\par
 Then, we can prove uniform (with respect to the regularizing parameter) a priori estimates
corresponding to \eqref{Linf bound b}, \eqref{estimate b}
and pass to the limit when that parameter tends to $0$.
\par
For \eqref{Linf bound b}, this is done by looking at the equation satisfied by $b(t,x)\, \exp( - \int_0^t r(s,x)\, ds)$ and by applying the maximum principle. For \eqref{estimate b}, the estimate is obtained thanks to the use of the multiplier $\Delta b$.

\bigskip

\noindent
\textbf{Acknowledgment}\\
\noindent
EB is a member of the INdAM-GNAMPA group.


\bigskip

\begin{thebibliography}{99}
\bibitem{Brezis}
\newblock{Brezis H.}, {\em Functional analysis, Sobolev spaces and partial differential equations}, Springer, Vol. 2. (2011).

\bibitem{Capone}
\newblock{Capone F., and Fiorentino L.}, {\em Turing instability for a Leslie–Gower model}, Ricerche di Matematica (2023).
%
\bibitem{Brocchieri}
\newblock{Brocchieri, E.}, {\em Evolutionary dynamics of populations structured by dietary diversity and starvation: cross-diffusion systems}, PhD thesis, Université Paris-Saclay; Università degli studi La Sapienza (Rome). Dipartimento di matematica (2023).
%
\bibitem{ChoKim}
\newblock{Cho, E., and Kim, Y.-J.},
{\em Starvation driven diffusion as a survival strategy of biological organisms},
Bull. Math. Biol. {\bf 75} (2013), 845--870.
%
\bibitem{Conf-Des-Sores}
\newblock{Conforto, F., Desvillettes, L., and Soresina, C.},
{\em About reaction-diffusion systems involving the Holling-type II and the Beddington-De Angelis functional responses for predator prey models},
Nonlinear Differ. Equ. Appl. {\bf 25}  (2018).
%
\bibitem{Desvillettes}
\newblock{Desvillettes, L.}, {\em About the triangular Shigesada–Kawasaki–Teramoto reaction cross diffusion system.} Ricerche di Matematica. {\bf 73} Suppl 1 (2024), 105-114.
%
\bibitem{DesTres}
\newblock{Desvillettes, L., and Trescases, A.},
{\em New results for triangular reaction cross diffusion system},
J. Math. Anal. Appl. {\bf 430} (2015), 32--59.
%
\bibitem{Des-Sores}
\newblock{Desvillettes L, and Soresina C.}, {\em Non-triangular cross-diffusion systems with predator–prey reaction terms.} Ricerche di Matematica. {\bf 68} (1) (2019), 295--314.
%
\bibitem{Evans}
\newblock{Evans, L. C.},  {\em Partial differential equations.} American Mathematical Society, Vol. 19 (2022).
%
\bibitem{Lombardo2013}
\newblock{Gambino, G., Lombardo, M. C., and Sammartino, M.}, {\em Pattern formation driven by cross-diffusion in a 2D domain}, Nonlinear Analysis: Real World Applications, {\bf 14} (3) (2013), 1755-1779.
%
\bibitem{Lombardo2012}
\newblock{Gambino, G., Lombardo, M. C., and Sammartino, M.}, {\em Turing instability and traveling fronts for a nonlinear reaction–diffusion system with cross-diffusion}, Mathematics and Computers in Simulation, {\bf 82} (6) (2012), 1112--1132.
%
\bibitem{Lombardo2008}
\newblock{Gambino, G., Lombardo, M. C., and Sammartino, M.}, {\em Cross-diffusion driven instability for a Lotka-Volterra competitive reaction–diffusion system}. In Waves and Stability in Continuous Media (2008), 297--302.
%
\bibitem{IMN}
\newblock{Iida, M., Mimura, M., and Ninomiya, H.},
{\em Diffusion, Cross-diffusion and Competitive Interaction},
J. Math. Biol. {\bf 53} (2006), 617--641.
%
\bibitem{KIM2014}
\newblock{Kim, Y. J., Kwon, O., and Li, F.}, {\em Global asymptotic stability and the ideal free distribution in a starvation driven diffusion.} J. Math. Biol. {\bf 68} (6) (2014), 1341--1370.
%
\bibitem{LSU}
\newblock{Ladyženskaja, O.A., Solonnikov, V.A., and Ural’ceva, N.N.}, {\em Linear and quasi-linear equations of parabolic type}, American Mathematical Soc., Vol 23 (1988).
%
\bibitem{Lou1998}
\newblock{Lou, Y., Ni, W.M., and Wu, Y.}, {\em On the global existence of a cross-diffusion system.} Discrete and Continuous Dynamical Systems, {\bf 4} (1998), 193--204.
%
%
\bibitem{Nirenberg}
\newblock{Nirenberg, L.}
{\em An extended interpolation inequality},
Annali della Scuola Normale Superiore di Pisa, Classe di Scienze {\bf 20.4} (1966), 733--737.
%
\bibitem{SKT}
\newblock{Shigesada, N., Kawasaki, K., and Teramoto, E.},
{\em Spatial segregation of interacting species},
J. Theor. Biol. {\bf 79} (1979), 83--99.
%
\bibitem{Yamada}
\newblock{Yamada, Y.}, {\em Global solutions for quasilinear parabolic systems with cross-diffusion effects}, Nonlinear Analysis: Theory, Methods and Applications, {\bf 24} (9) (1995), 1395--1412.
\end{thebibliography}
\bigskip

\noindent
Email addresses: \\
Elisabetta Brocchieri : elisabetta.brocchieri@uni-graz.at\\
Laurent Desvillettes : desvillettes@imj-prg.fr\\
Helge Dietert : helge.dietert@imj-prg.fr

\bigskip
\noindent ${^1}$
Departement of Mathematics and Scientific Computing, University of Graz, 8010 Graz, Austria.\\
\noindent ${^2}$
Université Paris Cité and Sorbonne Université, CNRS and IUF, IMJ-PRG, F-75006 Paris, France.\\
\noindent ${^3}$
Université Paris Cité and Sorbonne Université, CNRS, IMJ-PRG, F-75013 Paris, France.
\end{document}